\documentclass[a4paper,11pt,twoside]{article}

\usepackage{float}
\usepackage[utf8]{inputenc} 
\usepackage[T1]{fontenc}

\usepackage{charter}

\usepackage[dvipsnames]{xcolor}
\usepackage{verbatim}
\usepackage{listings}
\usepackage{multirow}

\usepackage[top=2.5cm, bottom=4.0cm, left=1.8cm, right=1.8cm]{geometry}
\usepackage{amsmath}
\usepackage{amsthm}	
\usepackage{amsfonts}	
\usepackage{amssymb	
	,bbm
	,units 
}
\usepackage{enumerate}
\usepackage[nobysame
,alphabetic
,initials
]{amsrefs}

\usepackage{algorithm}
\usepackage[noend]{algpseudocode}

\makeatletter
\renewcommand{\ALG@beginalgorithmic}{\footnotesize}
\makeatother

\usepackage{hyperref}
\usepackage{wrapfig}
\usepackage{graphicx}

\usepackage{
             ulem		
} 

\numberwithin{equation}{section}
\numberwithin{figure}{section}
\numberwithin{algorithm}{section}
\numberwithin{table}{section}
 \usepackage[nodayofweek]{datetime}

\renewcommand*{\thefootnote}{\fnsymbol{footnote}}

\title{Importance sampling for McKean-Vlasov SDEs}

\author{Greig Smith}
\author{
 	\normalsize Gon\c calo dos Reis \footnote{G. dos Reis acknowledges support from the \emph{Funda{\c c}$\tilde{\text{a}}$o para a Ci$\hat{e}$ncia e a Tecnologia} (Portuguese Foundation for Science and Technology) through the project UID/MAT/00297/2013 (Centro de Matem\'atica e Aplica\c c$\tilde{\text{o}}$es CMA/FCT/UNL).} \\[8pt]
         \small  University of Edinburgh, UK\\ 
         \small  and \\
	       \small  Centro de Matem\'atica e Aplica\c c$\tilde{\text{o}}$es, PT\\
         \small  G.dosReis@ed.ac.uk
 \and
         \normalsize Greig Smith \footnote{G. Smith was supported by The Maxwell Institute Graduate School in Analysis and its
Applications, a Centre for Doctoral Training funded by the UK Engineering and Physical
Sciences Research Council (grant EP/L016508/01), the Scottish Funding Council, the University of Edinburgh and Heriot-Watt
University.}  \\[8pt]
         \small  University of Edinburgh\\
	\small Maxwell Institute for Mathematical Sciences\\
	\small School of Mathematics\\
         \small  G.Smith-13@sms.ed.ac.uk
 \and
         \normalsize Peter Tankov\footnote{The research of Peter
           Tankov was supported by the FIME Research Initiative \texttt{www.fime-lab.org}}  \\[8pt]
         \small  CREST--ENSAE \\
         \small  Palaiseau, France\\
         \small  peter.tankov@ensae.fr
}

\renewcommand{\thefootnote}{\arabic{footnote}}

\date{ \currenttime, \ddmmyyyydate\today\qquad{(File: \tt \jobname.tex})}

\theoremstyle{plain}
\newtheorem{theorem}{Theorem}[section]
\newtheorem{lemma}[theorem]{Lemma}
\newtheorem{proposition}[theorem]{Proposition}

\newtheorem{definition}[theorem]{Definition}

\newtheorem{remark}[theorem]{Remark}

\newtheorem{assumption}[theorem]{Assumption}

\newcommand{\bE}{\mathbb{E}}

\newcommand{\bH}{\mathbb{H}}

\newcommand{\bN}{\mathbb{N}}
\newcommand{\bP}{\mathbb{P}}
\newcommand{\bQ}{\mathbb{Q}}
\newcommand{\bR}{\mathbb{R}}
\newcommand{\bS}{\mathbb{S}}

\newcommand{\bW}{\mathbb{W}}

\newcommand{\cB}{\mathcal{B}}

\newcommand{\cE}{\mathcal{E}}
\newcommand{\cF}{\mathcal{F}}

\newcommand{\cN}{\mathcal{N}}

\newcommand{\cP}{\mathcal{P}}

\newcommand{\cU}{\mathcal{U}}

\newcommand{\cX}{\mathcal{X}}

\newcommand{\dd}{\mathrm{d}}

\definecolor{darkgreen}{rgb}{0,0.35,0}

\newcommand{\1}{\mathbbm{1}}

\begin{document}


\maketitle
\renewcommand*{\thefootnote}{\arabic{footnote}}

\begin{abstract}
This paper deals with the Monte-Carlo methods for evaluating
expectations of functionals of solutions to
McKean-Vlasov Stochastic Differential Equations (MV-SDE) with drifts of super-linear growth. 
We assume that the MV-SDE is approximated in the standard manner by means of an 
interacting particle system and propose two importance sampling (IS)
techniques to reduce the variance of the resulting Monte Carlo
estimator. In the \emph{complete measure change} approach, the IS measure
change is applied simultaneously in the coefficients and in the
expectation to be evaluated. In the \emph{decoupling} approach we
first estimate the law of the solution in a first set of simulations
without measure change and then perform a second set of simulations
under the importance sampling measure using the approximate solution
law computed in the first step. 

For both approaches, we use large deviations techniques to
identify an optimisation problem for the candidate measure change.
The decoupling approach yields a far simpler optimisation problem than
the complete measure change, however, we can reduce the complexity of
the complete measure change through some symmetry arguments. We
implement both algorithms for two examples coming from the Kuramoto
model from statistical physics and show that the variance of the importance sampling schemes is up to 3 orders of magnitude
smaller than that of the standard Monte Carlo. The computational cost
is approximately the same as for standard Monte Carlo for the complete
measure change and only increases by a factor of 2--3 for the decoupled
approach. We also estimate the propagation of chaos error and find
that this is dominated by the statistical error by one order of magnitude.

\end{abstract}

\medskip


%

%
%
%
%
\section{Introduction}

The aim of this paper is to develop efficient importance sampling
algorithms for computing the expectations of
functionals of solutions to McKean-Vlasov stochastic differential
equations (MV-SDE). MV-SDE are stochastic differential equations where
the coefficients depend on the law of the solution, typically written
in the following form:
$$
\dd X_{t} = b(t,X_{t}, \mu_t)\dd t + \sigma(t,X_{t}, \mu_t)\dd W_{t},
\quad X_0 = x,
$$
where $\mu_{t}$ denotes the law of the process $X$ at time $t$, and
$W$ is a standard Brownian motion. MV-SDEs, also known as mean-field
equations, were originally introduced in physics to describe the
movement of an individual particle amongst a large number of
indistinguishable particles interacting through their mean field. They
are now used in a variety of other domains, such as finance,
economics, biology, population dynamics etc. 

Development of algorithms for the simulation of MV-SDEs is a very
active area of research. One of the earliest works to consider the
error and computational complexity involved in simulating a MV-SDE was
\cite{BossyTalay1997}. More recently \cite{GobetPagliarani2018},
\cite{SzpruchTanTse2017} and \cite{CrisanMcMurray2017} among others
(see references therein) developed more efficient methods for
simulating MV-SDEs under Lipschitz coefficients or stronger settings. 

A common technique for the simulation of MV-SDEs is to use the
interacting particle representation. Namely, we consider $i=1, \dots, N$ particles, where each $X^{i,N}$ satisfies the SDE with ${X}_{0}^{i,N}=x_0$
\begin{align}
\label{Eq:MV-SDE Propagation Intro}
\dd {X}_{t}^{i,N} 
= b\Big(t,{X}_{t}^{i,N}, \mu^{X,N}_{t} \Big) \dd t 
+ \sigma\Big(t,{X}_{t}^{i,N} , \mu^{X,N}_{t} \Big) \dd W_{t}^{i}, \quad
\mu^{X,N}_{t}(\dd x) := \frac{1}{N} \sum_{j=1}^N \delta_{X_{t}^{j,N}}(\dd x)
\end{align}
where $\delta_{{X}_{t}^{j,N}}$ is the Dirac measure at point
${X}_{t}^{j,N}$, and the Brownian motions $W^i, i=1,\dots,N$ are
independent. The so-called propagation of chaos result (see, e.g., \cite{Carmona2016Lectures}) states that
under sufficient conditions, as
$N\to \infty$, for every $i$, the process $X^{i,N}$ converges to
$X^i$, the solution of the MV-SDE driven by the Brownian motion
$W^i$. 

The system \eqref{Eq:MV-SDE Propagation Intro} is a system of
ordinary SDE and can be discretized with one of the many available
methods such as the Euler scheme. Let $X^{i,N,n}_t$ be the $i$-th
component of the solution of \eqref{Eq:MV-SDE Propagation Intro},
discretized on $[0,T]$ over $n$ steps. The quantity of interest,
which, in our case is $\theta = \mathbb E[G(X)]$, will then be approximated by
the Monte Carlo estimator
$$
\hat \theta^{N,n} = \frac{1}{N} \sum_{i=1}^N G(X^{i,N,n}). 
$$
The precision of this approximation is affected by three sources of
error.
\begin{itemize}
\item The statistical error, that is the difference between $\hat
  \theta^{N,n} $ and $\mathbb E[G(X^{i,N,n})]$. 
\item The discretization error, that is, the difference between
  $\mathbb E[G(X^{i,N,n})]$ and $\mathbb E[G(X^{i,N})]$.
\item The propagation of chaos error of approximating the MV-SDE with the interacting
  particle system, that is, the difference between $\mathbb
  E[G(X^{i,N})]$ and $\mathbb E[G(X)]$. 
\end{itemize}
The discretization error of ordinary SDEs has been analyzed by many
authors, and it is well known that, e.g., under the Lipschitz
assumptions the Euler scheme has weak convergence error of order
$\frac{1}{n}$. It is of course well known, the standard deviation of the statistical error is of order of $\frac{1}{\sqrt{N}}$.

There has also been some work detailing the error from the
propagation of chaos as a function of $N$, essentially for $G$ and $X$
nice enough the weak error is also of the order $\frac{1}{\sqrt{N}}$, see
for example \cite{KohatsuOgawa1997} and \cite{Bossy2004} for further
details. In spite of this relatively slow convergence, many MV-SDEs
have a reasonably ``nice'' dependence on the law which makes the
particle approximation a good technique. On the other hand, one often
wants to consider \emph{rare events} in the context of the MV-SDE, and in this realm the statistical error will dominate the propagation of chaos error. The focus of this paper is therefore on the statistical error
of the Monte Carlo method. In view of the poor convergence of the
standard Monte Carlo, it is typical to enhance the standard approach
with a so-called \emph{variance reduction} technique. Importance
sampling, which is the focus of this paper, is one such technique. We will discuss the point of statistical against propagation of chaos error in more detail in Section \ref{sec:Numerics}.

Importance sampling is based on the following identity, valid for any probability measure $\mathbb Q$ (absolutely continuous with respect to $\mathbb P$) 
$$
\mathbb E[G(X)] = \mathbb E_{\mathbb Q}\left[\frac{d\mathbb
    P}{d\mathbb Q}G(X)\right].
$$
The variance of the Monte Carlo estimator obtained by simulating $X$
under the measure $\mathbb Q$ and correcting by the corresponding
Radon-Nikodym density is different from that of the standard
estimator, and can be made much smaller by a judicious choice of the
sampling measure $\mathbb Q$. 

Importance sampling is most effective in the context of \emph{rare
  event simulation}, e.g., when the probability $\mathbb P[G(X)> 0]$
is small. Since the theory of large deviations is concerned with the
study of probabilities of rare events, it is natural to use measure
changes appearing in or inspired by the large deviations theory for
importance sampling. We refer, e.g., to \cite{DupuisWang2004}
and references therein for a review of this approach and to
\cites{GlassermanEtAl1999,GuasoniRobertson2008,Robertson2010}
for specific applications to financial models.  The large deviations
theory, on the one hand, simplifies the computation of the candidate
importance sampling measure, and on the other hand, allows to define
its optimality in a rigorous asymptotic framework.

The main contribution of this paper is two-fold. Firstly we show how
one can apply a change of measure to MV-SDEs, and propose
two algorithms that can carry this out: the \emph{complete measure change}
algorithm and the \emph{decoupling} algorithm. In the complete measure change approach, the IS measure
change is applied simultaneously in the coefficients and in the
expectation to be evaluated. In the decoupling approach we
first estimate the law of the solution in a first set of simulations
without measure change and then perform a second set of simulations
under the importance sampling measure using the approximate solution
law computed in the first step. 

Secondly, for both approaches, we use large deviations techniques to obtain an optimisation problem for the candidate measure change. We focus on the class of
Cameron-Martin transforms, under which the measure change is given by 
\begin{align}
\label{Eq:Radon Nikodym Derivative}
\frac{d\mathbb Q}{d\mathbb P}\Big|_{\mathcal F_T} = 
\cE\Big( \int_{0}^{T} f_{t} \dd W_{t}\Big) :=
\exp\left(  \int_{0}^{T} f_{t} \dd W_{t} - \frac{1}{2} \int_{0}^{T} f_{t}^{2} \dd t \right),
\end{align}
where $f_t$ is a deterministic function. Following earlier works on
the subject, we use the large deviations theory to construct a
tractable proxy for the variance of $G(X)$ under the new
measure. Of course, the presence of the interacting particle
approximation introduces additional complexity at this point. Moreover, unlike
the work of \cite{GuasoniRobertson2008} which considered a very
restrictive class of SDEs (the geometric Brownian motion), here we deal with a  general class of
MV-SDE where the drifts are of super-linear growth and satisfy a
monotonicity type condition. This is very important in practice
since many MV-SDEs fall into this category. 

We then minimise the large deviations proxy to obtain a candidate
optimal measure change for the two approaches that we consider. 
We find that the decoupling approach yields an easier optimisation
problem than the complete measure change, which results in a high
dimensional problem. However, by using exchangeability arguments the
latter problem can be transformed into a far simpler two dimensional one. 
We implement both algorithms for two examples coming from the Kuramoto
model from statistical physics and show that
the variance of
the importance sampling schemes is up to 3 orders of magnitude
smaller than that of the standard Monte Carlo. Moreover, the
computational only increases by a factor of 2--3 for the decoupling
approach and is approximately the same as standard Monte Carlo for the
complete measure change. We also estimate the propagation of chaos
error and find that this is dominated by the statistical error by one order of magnitude.
That being said, although the complete measure change appears to operate well in certain situations, it does rely on a change of measure which isn't too ``large''. We come back to this point throughout. 
\medskip 

Concerning the measure change paradigm, in this work we focus on
deterministic (open loop) measure changes over stochastic (feedback)
measure changes. This is a decision one faces when using importance
sampling and there are advantages and disadvantages to both. As
pointed out in \cite{GlassermanWang1997}, deterministic measure
changes may lead to detrimental results in terms of variance
reduction, however, the increase in computational time of the IS is
overall negligible. Stochastic measure changes as discussed in
\cite{DupuisWang2004} give improved variance reduction in far more
generality, however, calculating the measure change is computationally
burdensome, so the overall computational gain is less clear. As this
is the first paper to marry importance sampling with MV-SDEs we feel
it is beneficial to use deterministic based measure changes and leave
stochastic measure changes as interesting future work. We provide
precise conditions under which our deterministic measure change leads to an
asymptotically optimal importance sampling estimator in the class of all possible
measure changes.  Further, one of our algorithms requires a measure changed propagation of chaos result to hold (Proposition \ref{chaoscomplete}) and it is not clear how to prove such a result if one uses stochastic measure changes.
\medskip

The manuscript is organized as follows. In Section
\ref{Sec:Representation} we gather the preliminary results. In Section
\ref{sec:HowToIS} we discuss how importance sampling and measure
changes can be carried out for MV-SDE, and in Section \ref{Sec:Optimal
  Importance Sampling} we introduce our concept of optimality and identify the candidate optimal measure
changes using the theory of large deviations. Section \ref{sec:Numerics} illustrates numerically our results while proofs from Section \ref{Sec:Optimal Importance Sampling} are carried out in Section \ref{Sec:Proofs}.

\paragraph*{Acknowledgements}
The authors would like to thank Daniel Lacker (Columbia University), for the helpful discussion.

%
%
%
%

\section{Preliminaries}
\label{Sec:Representation}

Throughout the paper we work on a filtered probability space $(\Omega,
\cF, (\cF_{t})_{t \ge 0},  \bP)$ satisfying the usual conditions,
where $\cF_{t}$ is the augmented filtration of a standard multidimensional Brownian
motion $W$.
  
We consider some finite terminal time $T < \infty$ and use the following notation for spaces, which are standard in the McKean-Vlasov literature (see \cite{Carmona2016Lectures}). We define $\bS^{p}$ for $p \ge 1$, as the space of $\bR^{d}$-valued, $\cF_{\cdot}$-adapted processes $Z$, that satisfy, $\bE [ \sup_{0 \le t \le T} |Z(t)|^{p}]^{1/p} < \infty$. Similarly, $L_{t}^{p}(\bR^d)$, defines the space of $\bR^{d}$-valued, $\cF_{t}$-measurable random variables $X$, that satisfy, $\bE [|X|^{p}]^{1/p} < \infty$.

We will work with $\bR^d$, the $d$-dimensional Euclidean space of real numbers, and for $a=(a_1,\cdots,a_d)\in\bR^d$ and $b=(b_1,\cdots,b_d)\in\bR^d$ we denote by $|a|^2=\sum_{i=1}^d a_{i}^{2}$ the usual Euclidean distance on $\bR^d$ and by $\langle a,b\rangle=\sum_{i=1}^d a^i b^i$ the usual scalar product. 

Given the measurable space $(\bR^{d}, \cB (\bR^{d}))$, we denote by
$\cP(\bR^{d})$ the set of probability measures on this space, and
write $\mu \in \cP_{2}(\bR^{d})$ if $\mu \in \cP(\bR^{d})$ and for
some $x \in \bR^{d}$, $\int_{\bR^{d}} |x-y|^{2} \mu(\dd y) <
\infty$. We then have the following metric on the space $\cP_2(\mathbb
R^d)$ (Wasserstein metric) for $\mu, ~ \nu \in \cP_{2}(\bR^{d})$ (see \cite{dosReisSalkeldTugaut2017}),
\begin{align*}
W^{(2)}(\mu, \nu) 
=
\inf_{\pi} \left\{
\left( \int_{\bR^{d} \times \bR^{d}} |x-y|^2 \pi( \dd x, \dd y) \right)^{\frac12}
~ :
~ \pi \in \cP(\bR^{d} \times \bR^{d})
~ 
\text{with marginals $\mu$ and $\nu$}
\right\} \, .
\end{align*}

\subsection{McKean-Vlasov stochastic differential equations}

Let $W$ be an $l$-dimensional Brownian motion and take the progressively measurable maps $b:[0,T] \times \bR^d \times\cP_2(\bR^d) \to \bR^d$ and $\sigma:[0,T] \times \bR^d \times \cP_2(\bR^d) \to \bR^{d\times l}$.
MV-SDEs are typically written in the form,
\begin{equation}
\label{Eq:General MVSDE}
\dd X_{t} = b(t,X_{t}, \mu_t)\dd t + \sigma(t,X_{t}, \mu_t)\dd W_{t},  \quad X_{0} =x_{0},
\end{equation}
where $\mu_{t}$ denotes the law of the process $X$ at time $t$, i.e.~$\mu_t=\bP\circ X_t^{-1}$. Consider the following assumption on the coefficients.
\begin{assumption}
\label{Ass:Monotone Assumption}
	Assume that $\sigma$ is Lipschitz in the sense that there
        exists $ L>0$ such that for all $t \in[0,T]$ and all $x, x'\in \bR^d$ and $\forall \mu, \mu'\in \cP_2(\bR^d)$ we have that
	$$
	 |\sigma(t, x, \mu)-\sigma(t, x', \mu')|\leq L(|x-x'| + W^{(2)}(\mu, \mu') ),
	$$and let $b$ satisfy
	\begin{enumerate}
		\item One-sided Lipschitz growth condition in $x$ and Lipschitz
                  in law: there exists $ L>0$ such that for all $t
                  \in[0,T]$, all $ x, x'\in \bR^d$ and all $\mu, \mu'\in \cP_2(\bR^d)$ we have that
		$$
		\langle x-x', b(t, x, \mu)-b(t, x',\mu) \rangle \leq L|x-x'|^{2}
		\quad \text{and} \quad
		|b(t, x, \mu)-b(t, x,\mu')| \le  W^{(2)}(\mu, \mu') .
		$$
		
		\item Locally Lipschitz with polynomial growth in $x$:
                  there exists $q \in \bN$ with $q>1$ such that for
                  all  $t \in [0,T]$, $\forall \mu \in
                  \cP_{2}(\bR^{d})$ and all $x, ~ x' \in \bR^{d}$ the following holds.
		 $$
		 |b(t, x, \mu)-b(t, x',\mu)| \leq L(1+ |x|^{q} + |x'|^{q}) |x-x'|  .
		 $$
	\end{enumerate}
\end{assumption} 

Under these assumptions, an existence and uniqueness result for the
solution of the MV-SDE is given in \cite{dosReisSalkeldTugaut2017}. Note that this can be generalised to include random initial conditions.
\begin{theorem}[\cite{dosReisSalkeldTugaut2017}*{Theorem 3.3}]
\label{Thm:MV Monotone Existence}
	Suppose that $b$ and $\sigma$ satisfy Assumption \ref{Ass:Monotone Assumption} and be continuous in time.
	Further, assume for some $m \ge 2$, $X_{0} \in L_{0}^{m}(\bR^{d})$.
	Then there exists a unique solution for $X\in \bS^{m}([0,T])$ to the MV-SDE \eqref{Eq:General MVSDE}. 
	For some positive constant $C$  we have
	\begin{align*}
	\mathbb E \big[ \sup_{t\in[0,T]} |X_{t}|^{m} \big] 
	\leq C \Big(\bE[|X_0|^m]
	+ \Big(\int_{0}^{T}b(t,0, \delta_{0}) \dd t \Big)^{m}
	+ \Big(\int_{0}^{T}\sigma(t,0, \delta_{0})^{2} \dd t \Big)^{m/2}
	\Big) e^{C T}.
	\end{align*} 
\end{theorem}

\subsection{Large Deviation Principles}
In this section, we state the main results from the large deviations theory that we use throughout, for a full exposition the reader can consult texts such as \cite{DemboZeitouni2010} or \cite{DupuisEllis2011}.
The large deviation principle (LDP) characterizes the limiting behaviour,
as $\epsilon \rightarrow 0$, of a family of probability measures $\{
\mu_{\epsilon}\}$ in exponential scale on the space $(\cX ,
\cB_{\cX})$, with $\cX$ a topological space so that open and closed
subsets of $\cX$ are well-defined, and $\cB_{\cX}$ is the Borel $\sigma$-algebra on $\cX$. The limiting behaviour is defined via a so-called rate function. We assume the probability spaces have been completed, consequently, $\cB_{\cX}$ is the complete Borel $\sigma$-algebra on $\cX$. We have the following definition \cite{DemboZeitouni2010}*{pg.4}.
\begin{definition}
	[Rate function]
	A rate function $I$ is a lower semicontinuous mapping $I : \cX \rightarrow
	[0, \infty]$ (such that for all $\alpha \in [0,\infty)$, the level set $\Psi_{I}(\alpha)
	:= \{x : I(x) \le \alpha \}$ is
	a closed subset of $\cX$). A good rate function is a rate function for which all
	the level sets $\Psi_{I}(\alpha)$ are compact subsets of $\cX$. The effective domain of $I$,
	denoted $D_{I}$, is the set of points in $\cX$ of finite rate, namely, $D_{I} :=
	\{ x : I(x) < \infty \}$.
\end{definition}

We use the standard notation: for any set $\Gamma$, $\overline{\Gamma}$ denotes the closure and $\Gamma^{o}$ denotes the interior of $\Gamma$. As is standard practice in LDP theory, the infimum of a function over an empty set is interpreted as $\infty$. We then define what it means for this sequence of measures to have an LDP \cite{DemboZeitouni2010}*{pg.5}.
\begin{definition}
	A family of probability measures, $\{ \mu_{\epsilon}\}$ with $\epsilon >0$ satisfies the large deviation principle with a rate function $I$ if, for all $\Gamma \in \cB$,
	\begin{align}
	-\inf_{x \in \Gamma^{o}} I(x)
	\le 
	\liminf_{\epsilon \rightarrow 0} \epsilon \log \mu_{\epsilon}(\Gamma)
	\le
	\limsup_{\epsilon \rightarrow 0} \epsilon \log \mu_{\epsilon}(\Gamma)
	\le
	- \inf_{x \in \overline{\Gamma}} I(x)
	\, .
	\end{align}
\end{definition}
It is also typical to have LDP defined in terms of a sequence of random variables $Z_{\epsilon}$, in which case one replaces $\mu_{\epsilon}(\Gamma)$ by $\bP[ Z_{\epsilon}\in \Gamma]$.

The following result can be viewed as a generalisation of Laplace's approximation of integrals to the infinite dimensional setting and transfers the LDP from probabilities to expectations (see \cite{DemboZeitouni2010}).
\begin{lemma} [Varadhan's Lemma]
	\label{Lem:Varadhan}
	Let $\{\mu_{\epsilon}\}$ be a family of measures that satisfies a large deviation principle with good rate function $I$. Furthermore, let $Z_{\epsilon}$ be a family of random variables in $\cX$ such that $Z_{\epsilon}$ has law $\mu_{\epsilon}$ and let $\varphi : \cX \rightarrow \bR$ be any continuous function that satisfies the following integrability (moments) condition for some $\gamma >1$,
	\begin{align*}
	\limsup_{\epsilon \rightarrow 0} \epsilon \log \bE 
	\left[
	\exp \left( \frac{\gamma}{\epsilon} \varphi (Z_{\epsilon}) \right)
	\right]
	<
	\infty 
	\, .
	\end{align*} 
	Then,
	\begin{align*}
	\lim_{\epsilon \rightarrow 0} \epsilon \log \bE 
	\left[
	\exp \left( \frac{1}{\epsilon} \varphi (Z_{\epsilon}) \right)
	\right]
	=
	\sup_{x \in \cX} \left\{ \varphi(x)-I(x) \right \} \, .
	\end{align*} 
\end{lemma}
As is discussed in \cite{GuasoniRobertson2008}, one needs a slight extension to Varadhan's lemma to allow the function $\varphi$ to take the value $-\infty$. The extension is proved in \cite{GuasoniRobertson2008}.
\begin{lemma}
	Let $\varphi : \cX \rightarrow [-\infty, \infty)$ and assume the conditions in Lemma \ref{Lem:Varadhan} are satisfied. Then the following bounds hold for any $\Gamma \in \cB$
	\begin{align*}
	\sup_{x \in \Gamma^{0}} \{ \varphi(x)-I(x) \}
	& \le
	\liminf_{\epsilon \rightarrow 0} \epsilon \log \left( \int_{\Gamma^{o}} \exp\left( \frac{1}{\epsilon} \varphi(Z_{\epsilon}) \right) \dd \mu_{\epsilon} \right)
	\\ &
	\le
	\limsup_{\epsilon \rightarrow 0} \epsilon \log \left( \int_{\overline{\Gamma}} \exp\left( \frac{1}{\epsilon} \varphi(Z_{\epsilon}) \right) \dd \mu_{\epsilon} \right)
	\le
	\sup_{x \in \overline{\Gamma}} \{ \varphi(x)-I(x) \} \, .
	\end{align*}
\end{lemma}
The previous lemma allows us to control the $\liminf$ and $\limsup$ of
the process even when they are not equal (as is the case in Varadhan's lemma).

\subsection{Importance Sampling and large deviations}
\label{sec:MinimiseVariance}

To motivate our approach  we recall ideas from the pioneering works \cite{GlassermanEtAl1999}, \cite{GuasoniRobertson2008} and \cite{Robertson2010} which establish a connection between large deviations and importance sampling.
Importance sampling uses the following idea. Consider the problem of estimating $\bE_{\bP}[G(X)]$ where $X$ is some random variable/process governed by probability $\bP$. Through Radon-Nikodym theorem we can rewrite this expectation under a new measure $\bQ$ weighted by the Radon-Nikodym derivative, thus $\bE_{\bP}[G(X)]=\bE_{\bQ}[G(X) \frac{\dd \bP}{\dd \bQ}]$. Although the expectations (first moments) are the same, the variance under $\bQ$ is,
\begin{align}
\label{Eq:Variance under Q}
\text{Var}_{\bQ} \Big[G(X) \frac{\dd \bP}{\dd \bQ} \Big]=
\bE_{\bP}\Big[ G(X)^{2} \frac{\dd \bP}{\dd \bQ} \Big]
-\bE_{\bP}\Big[G(X)\Big]^2 \, .
\end{align}
As it turns out, if one can choose $\frac{\dd \bQ}{\dd \bP} = \frac{G}{\bE_{\bP}[G]}$, then the variance under $\bQ$ is zero, i.e. we have no error in our Monte Carlo simulation. Unfortunately though, in order to choose such a change of measure one would need to a priori know the value of $\bE_{\bP}[G(X)]$ i.e. the value we wish to estimate in the first place.

Instead one typically chooses $\bQ$ to minimise \eqref{Eq:Variance
  under Q} over a set of equivalent probability measures, chosen to
add only a small amount of extra computation  and such that the process
$X$ is easy to simulate under the new measure. Specializing to the
Brownian filtration, a common choice of $\mathbb Q$ is the Girsanov
transform, \eqref{Eq:Radon Nikodym Derivative} where $f$ is often taken to be a deterministic function. 

For example in \cite{TengEtAl2016} the authors
develop an importance sampling procedure in the context of Gaussian
random vectors through a so-called ``tilting''
parameter, which  corresponds to
shifting the mean of the Gaussian random vector via a Girsanov transform. Although this method
is intuitive, it still requires estimation of the Jacobian
of $G$ w.r.t. the tilting parameter and applying Newton's method to
select the optimal parameter value. These steps can be computationally
expensive, and it is difficult to obtain rigorous optimality results.  

Even after one has reduced the set of measures $\bQ$ to optimise over,
in general the problem of minimizing \eqref{Eq:Variance
  under Q} will not have a closed form solution. Thus we instead minimize a proxy for the variance obtained in the so-called small noise asymptotic regime as discussed in
\cite{GlassermanEtAl1999} and \cite{GuasoniRobertson2008}. Assuming
that a Girsanov change of measure is used, we want to minimise 
\begin{align}
\label{girsvar}
\bE_{\bP}\left[ G(W)^{2} \frac{\dd \bP}{\dd \bQ} \right]
=
\bE_{\bP} \left[ 
\exp\left( 2 F(W) - \int_{0}^{T} f_{t} \dd W_{t} + \frac{1}{2} \int_{0}^{T} f_{t}^{2} \dd t \right)
\right],\quad \textrm{with } F=\log (G).
\end{align}
Typically $G$ is defined as a functional of the SDE,
but here with a slight abuse of notation we have redefined it as the
functional of the driving Brownian motion. It is important for this
type of argument that we are able to write the solution of the SDE in
terms of BM as well, i.e. we can write $X_{t}=H(t,
W_{\cdot})$. Finding the optimal $f$ by minimizing \eqref{girsvar} is
in general intractable, hence an asymptotic approximation of the
variance should be constructed. Let us consider, 
\begin{align*}
\epsilon \log \bE_{\bP} \left[ 
\exp\left( \frac{1}{\epsilon}\left(2 F(\sqrt{\epsilon} W) - \int_{0}^{T} \sqrt{\epsilon} f_{t} \dd W_{t} + \frac{1}{2} \int_{0}^{T} f_{t}^{2} \dd t \right) \right)
\right] \, ,
\end{align*}
which equals $\log$ of \eqref{girsvar} when $\epsilon=1$, the \emph{small noise asymptotic approximation} is then,
\begin{align*}
L(f) :=
\limsup_{\epsilon \rightarrow 0}
\epsilon \log \bE_{\bP} \left[ 
\exp\left( \frac{1}{\epsilon}\left(2 F(\sqrt{\epsilon} W) - \int_{0}^{T} \sqrt{\epsilon} f_{t} \dd W_{t} + \frac{1}{2} \int_{0}^{T} f_{t}^{2} \dd t \right) \right)
\right] \,.
\end{align*}
One then computes a candidate variance reduction parameter $f^*$ by
minimizing $L(f)$, which can be thought of as approximating
$\bE_{\bP}\left[ G(W)^{2} \frac{\dd \bP}{\dd \bQ} \right]$ by
$\exp(L(f))$. Crucially, $L$ is in a form that can be evaluated using
the Varadhan's lemma, i.e.,  we can change $L$ into a supremum
depending on the rate function. The
parameter $f^*$, which minimises $L$ over some predefined space is known as \emph{asymptotically optimal}, see \cite{GuasoniRobertson2008}. We will give a precise definition of this concept later.
It is important to note, these approximations are not approximations for the original problem (calculate $\bE_{\bP}[G(X)]$), they are only approximations to help choose the change of measure we want to apply.

%
%
%
%

\section{Importance sampling for MV-SDEs}
\label{sec:HowToIS}

Leaving LDPs and the optimality of the IS (importance sampling) on the
side, let us discuss how IS can be achieved for MV-SDEs with a given
measure change.

Recall that MV-SDEs take the form \eqref{Eq:General MVSDE}. Because we change the measure we make explicit the dependence on the law of the solution process $\mu^X_{t,\bP} = \bP\circ X_t^{-1}$. 
 If one knows the law $\mu^X$ beforehand , then one can treat the
 MV-SDE as a ``standard'' SDE and use IS as usual. However, typically
 one does not have access to the law, and the MV-SDE must be
 approximated by a so-called particle system approximation. 

\textbf{The interacting particle system approximation.} We approximate \eqref{Eq:General MVSDE} (driven by the $\bP$-Brownian motion $W^\bP$), using an $N$-dimensional system of interacting particles. Let $i=1, \dots, N$ and consider $N$ particles $X^{i,N}$ satisfying the SDE with ${X}_{0}^{i,N}=x_0$ 
\begin{align}
\label{Eq:MV-SDE Propagation}
\dd {X}_{t}^{i,N} 
= b\Big(t,{X}_{t}^{i,N}, \mu^{X,N}_{t} \Big) \dd t 
+ \sigma\Big(t,{X}_{t}^{i,N} , \mu^{X,N}_{t} \Big) \dd W_{t}^{i, \bP}, \quad
\mu^{X,N}_{t}(\dd x) := \frac{1}{N} \sum_{j=1}^N \delta_{X_{t}^{j,N}}(\dd x)
\end{align}
where $\delta_{{X}_{t}^{j,N}}$ is the Dirac measure at point
${X}_{t}^{j,N}$, and the independent $\bP$-Brownian motions $W^{i,\bP}, i=1,\dots,N$ (also independent of the BM $W^\bP$ appearing in \eqref{Eq:General MVSDE}). Due to the several changes of the measure throughout this section we keep track of which $W$ we refer to. 
\begin{remark}[On the empirical measure $\mu^{X,N}_{t}$]
Unlike standard measures, empirical measures do not have dependence on the underlying measure $\bP$, namely empirical measures are maps that depend on a sequence of $\omega^{i} \in \Omega$, thus one should write $\mu^{X,N}_{t}$ instead of $\mu^{X,N}_{t,\bP}$. 
Of course, this is a pathwise statement, since the $\omega^{i}$ are generated under $\bP$, the \emph{distribution} of the empirical measure does depend on $\bP$.
\end{remark}

\textbf{Propagation of chaos.}
In order to show that the particle approximation is of use, one shows a so-called propagation of chaos result. Although different types exist a common one is the pathwise convergence result where we consider the system of non interacting particles  
\begin{align}
	\label{Eq:Non interacting particles}
	\dd X_{t}^{i} = b(t, X_{t}^{i}, \mu^{X^{i}}_{t,\bP}) \dd t + \sigma(t,X_{t}^{i}, \mu^{X^{i}}_{t,\bP}) \dd W_{t}^{i, \bP}, \quad X_{0}^{i}=x_0 \, ,\quad t\in [0,T] \, ,
\end{align}
which are of course just MV-SDEs and since the $X^{i}$s are
independent, then $\mu^{X^{i}}_{t,\bP}=\mu^{X}_{t,\bP}$ $\forall
i$. Under sufficiently nice conditions, one can then prove the
following convergence result (see \cite{Carmona2016Lectures}*{Theorem 1.10} for
example)
\begin{align*}
\lim_{N \rightarrow \infty} \sup_{1 \le i \le N}
\bE_{\bP} \left[ 
\sup_{0 \le t \le T} |X_{t}^{i,N} - X_{t}^{i}|^{2}
\right] = 0 \, .
\end{align*}
Note that, for all SDEs appearing below, we have initial condition $x_{0}$ and work on the interval $[0,T]$.

\textbf{Setup to change measures.} When it comes to changing the measure under which we simulate we are also changing our approximation of the law. Since MV-SDEs depend explicitly on the law, this makes importance sampling more difficult. This will be one of the main points throughout this section.

Fix a deterministic square-integrable function  $\dot h\in
L_{0}^2(\bR)$.
 Then one can define
the probability measure $\bQ$ via the Girsanov transform 
$\frac{d \bQ}{d\bP}|_{\cF_{T}}:=\cE(\int_{0}^{T} \dot{h}_{t} \dd
W_{t}^{\bP})$, see \eqref{Eq:Radon Nikodym Derivative}, so that $\dd W^{\bQ}_t= \dd W_{t}^{\bP}-\dot h_t \dd t$
is a $\mathbb Q$-Brownian motion. We note that the Radon-Nikodym density $\frac{\dd \bQ}{\dd
  \bP}|_{\cF_{t}}=\cE(\int_0^\cdot \dot h_s \dd W_s^{\bP})_t=:\cE_t$ is itself the
solution of the SDE
\begin{align*}
\dd \cE_{t} = \dot h_t \cE_{t} \dd W_t^{\bP}, \quad \cE_{0}=1
\quad\Rightarrow\quad
\cE_{t}=\exp\Big\{ \int_0^t \dot h_s \dd W_s^{\bP} - \frac12 \int_0^t  |\dot h_s|^2 \dd s \Big\}.
\end{align*}
Since $\bP$ and $\bQ$ are equivalent, one can also define $Z_{t}:=
\cE_{t}^{-1} :=\frac{\dd \bP}{\dd \bQ}|_{\cF_{t}}$. With our conditions on $\dot{h}$ it is also a straightforward task to show $\cE_{t}$ and $Z_{t}$ are in $\bS^{p}$ for all $p \ge 1$.

Recall our goal: \emph{estimate
  $\bE_{\bP}[G(X_{T})]=\bE_{\bQ}[G(X_{T}) \frac{\dd \bP}{\dd \bQ}]$
  for some function $G$ by simulating $X$ under $\bQ$}. In the
following paragraphs we present two alternative ways to achieve this
goal.

\textbf{A running example.} We present our algorithm in general setting with \eqref{Eq:General MVSDE}. For the sake of clarity and easiness of presentation, we often recourse to a particular class of MV-SDEs (under $\bP$),
\begin{align}
\label{Eq:Simple Example MV-SDE}
\dd X_{t} = \hat{b}\big(t, X_{t}, \bE_{\bP}[f(X_{t})]\big) \dd t + \sigma\dd W_{t}^{\bP}, \quad X_{0}=x_0 \, ,\quad t\in [0,T] \, .
\end{align}
with $\sigma\in\bR_{+}$ and $f,\hat{b}$ nice\footnote{We use $\hat{b}$ here since it takes the expectation rather than a measure input.}. We believe many of the arguments that are used at this level can be extended to cover more general MV-SDEs (such as higher order interactions). However, obtaining analogous results to those of standard MV-SDEs, such as propagation of chaos, is made more challenging by the inclusion of the measure change. Therefore, these have to be considered on a case by case basis.

\subsection{Fixing the Empirical Law - a decoupling argument}
\label{sec:DecouplingIS}

An obvious way to solve the problem of IS is to approximate the law of
the MV-SDE under $\mathbb P$ and use that as a fixed input to a new
equation which will be simulated under $\mathbb Q$. In this set up the McKean-Vlasov SDE turns into an SDE with random coefficients. The algorithm is as follows.
\begin{enumerate}
	\item Use \eqref{Eq:MV-SDE Propagation} with $N$ particles to approximate \eqref{Eq:General MVSDE}.
	 Use some numerical scheme (under $\bP$, say Euler) to simulate the particles in time, calculating an empirical law over $[0,T]$. This gives an approximation for the empirical law $\mu^{N}_{t}$ which is then fixed. 
	
Define a new SDE, approximating the original MV-SDE \eqref{Eq:General MVSDE},  which is now a \emph{standard SDE with random coefficients}
\begin{align}
\label{eq:Decoupled SDE}
\dd \bar X_{t} 
= b(t, \bar X_{t}, \mu^{N}_{t}) \dd t 
+ \sigma(t,\bar X_{t}, \mu^{N}_{t}) \dd W_{t}^{\bP}, \quad \bar X_{0}=x_0,
\end{align}	
where $W^{\bP}$ is a $\bP$-Brownian motion independent of the $\{W^{i,\bP}\}_{i=1,\cdots,N}$ appearing in \eqref{Eq:MV-SDE Propagation}.
	SDEs with random coefficients appear typically in optimal control, hence the reader can consult texts such as \cite{YongZhou1999}*{Chapter 1} for further details on existence uniqueness of such SDEs.

	\item Change the probability measure to $\bQ$, which is our importance sampling measure change. Simulate \eqref{eq:Decoupled SDE} under this new measure, i.e.
\begin{align*}
\dd \bar{X}_{t} = \left( b(t, \bar{X}_{t}, \mu^{N}_{t}) + \dot{h}_{t} \sigma(t, \bar{X}_{t}, \mu^{N}_{t})  \right) \dd t 
+ \sigma(t, \bar{X}_{t}, \mu^{N}_{t}) \dd W_{t}^{\bQ}, \quad \bar{X}_{0}=x_0 \, .
\end{align*}

	\item This second run is therefore standard importance sampling, but the SDE has random coefficients i.e. the empirical law is random. 
\end{enumerate}

We will refer to algorithms of this form as \emph{Decoupling Algorithms}.
This scheme has the disadvantage in that it requires twice the amount of simulation and one will require a handle on the error coming from the original approximation of the law.

It is not a requirement to use interacting particles to approximate the law of the SDE, any approximation will work. The goal here is to make the SDEs independent.

\subsection{Complete Measure Change}

\label{Sec:Complete Measure Change}
An alternative is to change the measure under which we are simulating in the coefficients \emph{and} the Brownian motion. This is not a simple problem and as far as we are aware changing the measure of a MV-SDE and its particle approximation is not discussed elsewhere in the literature (for this purpose\footnote{Measures changes for MV-SDE appear in methods requiring to remove the drift altogether, for instance in establishing weak solutions to MV-SDEs, see e.g. \cite{DawsonGaertner1987-DG1987}.}), we therefore provide a discussion along with the pitfalls here. This is more complex than the decoupled case and for clarity we use \eqref{Eq:Simple Example MV-SDE} throughout.

The measure changed version of \eqref{Eq:Simple Example
    MV-SDE} takes the form,
\begin{align*}
\dd X_{t} &= \Big( \hat{b}(t, X_{t}, \bE_{\bP}[f(X_{t})]) + \sigma
            \dot{h}_{t} \Big) \dd t + \sigma\dd W_{t}^{\bQ} \\
&=\Big( \hat{b}(t, X_{t}, \bE_{\bQ} \Big[ f(X_{t}) Z_{t}\Big]) + \sigma \dot{h}_{t} \Big) \dd t + \sigma\dd W_{t}^{\bQ} .
\end{align*}
where again $Z := \cE^{-1}$.

In view of simulation, we re-write the measure changed MV-SDE as a system
\begin{align*}
\dd X_{t} &= \Big( \hat{b} \Big(t, X_{t}, \bE_{\bQ} \big[ F(X_t,Z_t)\big]\Big) + \sigma \dot{h}_{t} \Big) \dd t + \sigma\dd W_{t}^{\bQ},
\quad
\text{and}
\quad
\dd Z_{t} = \dot h_t Z_{t} \dd W_{t}^{\bQ}, \quad Z_0=1 \, ,
\end{align*}
where $F(x,y)=f(x) y$. We now write the interacting particle system for the pair $X,Z$ under
$\bQ$:
\begin{align}
\label{Eq:interacting particles-v01}
\dd X_{t}^{i,N} &= \Big( \hat{b} \big(t, X_{t}^{i,N}, \frac1N \sum_{j=1}^N F(X_{t}^{j,N},Z^{j,N}_t) \big) + \sigma \dot{h}_{t} \Big) \dd t + \sigma\dd W_{t}^{i,\bQ},
\\
\dd Z^{i,N}_t
&=\dot h_t Z^{i,N}  \dd W^{i,\bQ}_t, 
 \quad
 Z_{0}^{i,N} =1 \, . \notag
\end{align}
The importance sampling estimator of $\theta = \mathbb E^{\mathbb P}[G(X_T)]$ then takes
the form
\begin{align}
\hat\theta_h = \frac{1}{N}\sum_{i=1}^N Z^{i,N}_T G(X^{i,N}_T). \label{estimatorcomplete}
\end{align}
\begin{remark}
One may be tempted to use a different approach, namely first apply an
interacting particle approximation under $\mathbb P$ which yields
\begin{align*}
\dd X_{t}^{i,N} &= \hat{b} \big(t, X_{t}^{i,N}, \frac1N \sum_{j=1}^N f(X_{t}^{j,N}) \big) \dd t + \sigma\dd W_{t}^{i,\bP},
\end{align*}
and then change the measure for the particle system, writing
\begin{align*}
\dd X_{t}^{i,N} &= \Big(\hat{b} \big(t, X_{t}^{i,N}, \frac1N
                  \sum_{j=1}^N f(X_{t}^{j,N}) \big) + \sigma\dot h_t\Big)\dd t + \sigma\dd W_{t}^{i,\bQ},
\end{align*}
where we have taken the same $\dot h$ for every Brownian motion in
order for all particles to have the same law. However, it is easy to
see by the standard propagation of chaos result that as $N\to \infty$,
this particle system converges to the solution of the MV-SDE
$$
\dd X_{t} = \Big(\hat{b} \big(t, X_{t}, 
                  \mathbb E^{\mathbb Q}[X_t] \big) + \sigma\dot h_t\Big)\dd t + \sigma\dd W_{t}^{\bQ}=\hat{b} \big(t, X_{t}, 
                  \mathbb E^{\mathbb Q}[X_t] \big) \dd t + \sigma\dd W_{t}^{\bP},
$$
which is not what one is looking for. 

\end{remark}

To state a propagation of chaos result for the particle system \eqref{Eq:interacting particles-v01} we introduce the auxiliary system of non-interacting particles,
\begin{align}
\label{eq:non-interacting particles-v01}
	\dd X_{t}^{i} & =
	 \Big( \hat{b} \Big(t, X_{t}^{i}, \bE_{\bQ} \big[ F(X_{t}^{i},Z_{t}^{i})\big] \Big) + \sigma \dot{h}_{t} \Big) \dd t + \sigma\dd W_{t}^{i,\bQ}\, ,
	\\
	\nonumber
	\dd Z^{i}
	&= \dot h_t Z^{i} \dd W^{i,\bQ}_t, 
	\quad Z^{i}=1 \, .
\end{align}

\begin{proposition}\label{chaoscomplete}
	Consider the following measure changed MV-SDE (see \eqref{eq:non-interacting particles-v01}),
	\begin{align}
	\label{Eq:Measure changed example}
	\dd X_{t}^{i} & =
	\Big( \hat{b}\Big(t, X_{t}^{i}, \bE_{\bQ} \Big[ f(X_{t}^{i}) Z_{t}^{i} \Big]\Big) + \sigma \dot{h}_{t} \Big) \dd t + \sigma\dd W_{t}^{i,\bQ}\, ,
	\quad
	\dd Z_{t}^{i}
	= \dot h_t Z_{t}^{i} \dd W^{i,\bQ}_t, 
	\quad Z_{0}^{i}=1 \, ,
	\end{align}
	where $\hat{b}$ is continuous in time, $\hat{b}$ and $f$ are Lipschitz in space, and $\hat{b}$ is a bounded Lipschitz function in its third variable. Let $X_{t}^{i,N}$, denote the corresponding particle approximation (see \eqref{Eq:interacting particles-v01}).	
	 Then the following pathwise propagation of chaos result holds,
	\begin{align*}
	\lim_{N \rightarrow \infty} \sup_{1 \le i \le N}
	\bE_{\bQ} \left[  \sup_{0 \le t \le T} |X_{t}^{i,N} - X_{t}^{i}|^{2} \right]
	=0 \, .
	\end{align*}
\end{proposition} 
This proposition may be used to analyze the convergence of the Monte
Carlo estimator \eqref{estimatorcomplete}. Indeed, due to the fact that there is no coupling (or law dependency) in $Z_{t}^{i,N}$, $Z^{i,N}= Z^{i}$ and $\hat \theta_h$ can
be represented as follows.
$$
\hat \theta_h = \frac{1}{N}\sum_{i=1}^N Z^{i}_T G(X^{i}_T) + \frac{1}{N}\sum_{i=1}^N Z^{i}_T (G(X^{i,N}_T)
-G(X^{i}_T)
).
$$
The first term above converges to $\theta$ as $N\to \infty$ by the law of
large numbers, and the second term can be shown, e.g., to converge to
zero in probability using Proposition \ref{chaoscomplete} if $G$ is
sufficiently regular.  
\begin{proof}[Proof of Proposition \ref{chaoscomplete}]
	The idea of the proof is to appeal to a Gronwall type inequality, but this is made difficult due to the presence of $Z$ term in \eqref{Eq:Measure changed example}. One can note, due to the assumptions on the coefficients of the SDE, all $p$-moments exist. Using our prescribed form of the MV-SDE we obtain,
	\begin{align*}
	|X_{t}^{i,N} - X_{t}^{i}|^{2}
	\le
	C \int_{0}^{t} \big| \hat{b} \Big(s, X_{s}^{i,N}, \frac{1}{N} \sum_{j=1}^{N} f(X_{s}^{j,N}) Z_{s}^{j} \Big)
	- \hat{b} \left(s, X_{s}^{i}, \bE_{\bQ}[f(X_{s}^{i}) Z_{s}^{i}] \right) \big|^{2} \dd s \, .
	\end{align*}
	Let $s \in [0,T]$, then introduce the terms, $\hat{b} \Big(s,
        X_{s}^{i}, \frac{1}{N} \sum_{j=1}^{N} f(X_{s}^{j,N})Z_{s}^{j}
        \Big)$ and $\hat{b} \Big(s, X_{s}^{i}, \frac{1}{N}
        \sum_{j=1}^{N} f(X_{s}^{j})Z_{s}^{j} \Big)$, where the
        empirical measure in the second term is the one constructed
        from the i.i.d. SDEs in \eqref{Eq:Measure changed example},
        hence each $X^{j}$ corresponds to a independent realisation of
        the MV-SDE, namely it has the correct distribution. Splitting
        the original difference into three, we use the Lipshitz
        property in space for the first one, to obtain,
	\begin{align*}
	\big| \hat{b} \Big(s, X_{s}^{i,N}, \frac{1}{N} \sum_{j=1}^{N} f(X_{s}^{j,N}) Z_{s}^{j} \Big)
	- \hat{b} \Big(s, X_{s}^{i}, \frac{1}{N} \sum_{j=1}^{N} f(X_{s}^{j,N})Z_{s}^{j} \Big) \big|^{2}
	\le
	C|X_{s}^{i,N}-X_{s}^{i}|^{2} \, .
	\end{align*}
	For the second difference we use the fact that $\hat{b}$ is
        bounded along with the Lipschitz property in the third
        variable, which yields
	\begin{align*}
	&\big| \hat{b} \Big(s, X_{s}^{i}, \frac{1}{N} \sum_{j=1}^{N} f(X_{s}^{j,N})Z_{s}^{j} \Big)
	- \hat{b} \Big(s, X_{s}^{i}, \frac{1}{N} \sum_{j=1}^{N} f(X_{s}^{j})Z_{s}^{j} \Big) \big|^{2}
	\\
	&
	\le
	C\big| \hat{b} \Big(s, X_{s}^{i}, \frac{1}{N} \sum_{j=1}^{N} f(X_{s}^{j,N})Z_{s}^{j} \Big)
	- \hat{b} \Big(s, X_{s}^{i}, \frac{1}{N} \sum_{j=1}^{N} f(X_{s}^{j})Z_{s}^{j} \Big) \big|
	\le
		C  \frac{1}{N} \sum_{j=1}^{N} Z_{s}^{j}|X_{s}^{j,N} - X_{s}^{j}| \, .
	\end{align*}
	Finally, again from the Lipschitz property we obtain,
	\begin{align*}
	\big| \hat{b} \Big(s, X_{s}^{i}, \frac{1}{N} \sum_{j=1}^{N} f(X_{s}^{j})Z_{s}^{j} \Big)
	- \hat{b} \left(s, X_{s}^{i}, \bE_{\bQ}[f(X_{s}^{i}) Z_{s}^{i}] \right) \big|^{2}
	\le
	C
	\big|\frac{1}{N} \sum_{j=1}^{N} f(X_{s}^{j}) Z_{s}^{j} 
	- \bE_{\bQ}[f(X_{s}^{i}) Z_{s}^{i}] \big|^{2}	
	\, .
	\end{align*}
	Hence the following bound holds,
	\begin{align*}
	& \bE_{\bQ} \Big[ \sup_{0 \le t \le T} |X_{t}^{i,N} - X_{t}^{i}|^{2} \Big]
	\\
	&
	\le
	C \int_{0}^{T} \bE_{\bQ} [ |X_{s}^{i,N}-X_{s}^{i}|^{2}] 
	+
	 \frac{1}{N} \sum_{j=1}^{N} \bE_{\bQ}\Big[ Z_{s}^{j}|X_{s}^{j,N} - X_{s}^{j}| \Big]
	+
	\bE_{\bQ} \Big[ \big|\frac{1}{N} \sum_{j=1}^{N} f(X_{s}^{j}) Z_{s}^{j} 
	- \bE_{\bQ}[f(X_{s}^{i}) Z_{s}^{i}] \big|^{2} \Big]
	 \dd s \, .
	\end{align*}
	One can use Cauchy-Schwarz along with the properties of $Z$ to obtain,
	\begin{align*}
	\bE_{\bQ}\Big[ Z_{s}^{j}|X_{s}^{j,N} - X_{s}^{j}| \Big]
	\le
	C \bE_{\bQ}\Big[ |X_{s}^{j,N} - X_{s}^{j}|^{2} \Big]^{\frac12}
	\le
	C \bE_{\bQ}\Big[ \sup_{0 \le u \le s} |X_{u}^{j,N} - X_{u}^{j}|^{2} \Big]^{\frac12} \, .
	\end{align*}
	Although at first it appears one cannot use Gronwall here,
        there is a nonlinear generalisation due to Perov (see
        \cite{MitrinovicEtAl2012}*{Theorem 1, p360}) which we can use
        since the nonlinear term on the RHS is square root of the term
        on the left. Finally, take the supremum over $i$ and using the
        fact that the variables $f(X_{s}^{j}) Z_{s}^{j}$ are
        i.i.d. and square integrable, we obtain,
	\begin{align*}
	\sup_{1 \le i \le N}\bE_{\bQ} \Big[ \sup_{0 \le t \le T} |X_{t}^{i,N} - X_{t}^{i}|^{2} \Big]
	&\le
	C e^{C}\int_{0}^{T}
	\bE_{\bQ} \Big[ \big|\frac{1}{N} \sum_{j=1}^{N} f(X_{s}^{j}) Z_{s}^{j} 
	- \bE_{\bQ}[f(X_{s}^{i}) Z_{s}^{i}] \big|^{2} \Big]
	\dd s \, \\
&\leq \frac{ C e^{C}}{N}\int_{0}^{T}\bE_{\bQ} \Big[ \big| f(X_{s}^{1}) Z_{s}^{1} 
	- \bE_{\bQ}[f(X_{s}^{1}) Z_{s}^{1}] \big|^{2} \Big]
	\dd s \,\to 0
	\end{align*}
as $N\to \infty$, which concludes the proof. 	
\end{proof}

\subsubsection*{The Complete Measure Change Algorithm}
We now describe the algorithm for simulating a general MV-SDE under a
complete measure change.
\begin{enumerate}
	\item Simulate the following particle system for the MV-SDE
          after the measure change: 
	\begin{align*}
		\dd X_{t}^{i,N} 
		& = \left( b \left( t, X_{t}^{i,N}, \frac{1}{N} \sum_{j =1}^{N} Z_{t}^{j} \delta_{X_{t}^{j,N}} \right) 
		+ \dot{h}_{t}
		\sigma \left( t,X_{t}^{i,N}, \frac{1}{N} \sum_{j =1}^{N} Z_{t}^{j} \delta_{X_{t}^{j,N}} \right)  \right) \dd t 
		\\
		&
		\qquad \qquad +
		\sigma \left( t,X_{t}^{i,N}, \frac{1}{N} \sum_{j =1}^{N} Z_{t}^{j} \delta_{X_{t}^{j,N}} \right) \dd W_{t}^{i,\bQ}, 
		\\
		 \dd Z_{t}^{i}
		 &= 
		 \dot h_t Z_{t}^{i} \dd W^{i,\bQ}_t, 
		 \quad Z_{0}^{i}=1\, .
	\end{align*}

	\item Compute the importance sampling estimator using the
          following formula:
$$
\hat\theta_h = \frac{1}{N}\sum_{i=1}^N Z^{i,N}_T G(X^{i,N}_T).
$$

\end{enumerate}
We will refer to algorithms of this form as \emph{Complete Measure Change Algorithms}.
An advantage one can immediately see is that one simulates the
particles only once. A key disadvantage is that the importance sampling to estimate the object of interest $\bE[G(X_{T})]$, may yield a poorer estimation of the original law $\mu$ and the term $\bE_\bQ[f(X_t)Z_t]$ in \eqref{eq:non-interacting particles-v01}. We will discuss this in Section \ref{sec:Numerics}.

%
%
%
%

\section{Optimal Importance Sampling for McKean-Vlasov SDEs}
\label{Sec:Optimal Importance Sampling}
The previous section detailed algorithms for simulating MV-SDEs under an arbitrary change of measure. We now use the theory of large deviations to determine, in a certain optimal way, a measure change which will reduce the variance of the estimate. 

An important point here is that we will be using the LDP for Brownian
motion, rather than that for MV-SDEs. There are several works dealing with Large Deviations for MV-SDEs and their associated interacting particles systems, see \cite{BudhirajaDupuisFischer2012}, \cite{Fischer2014}, \cite{dosReisSalkeldTugaut2017} but such results are not of use here since we must be able to cheaply simulate the MV-SDE after the change of measure. We restrict to Girsanov measure changes since we know how the SDE changes under the measure change.

In this section we first show how the LDP framework can be applied to both algorithms to yield a simplified optimisation problem for finding the asymptotically optimal measure change (Theorems \ref{Thm:Complete Measure Change} and \ref{Thm:Decoupled}) and then demonstrate how these simplified optimization problems can  be solved in practice.

\subsection{Preliminaries}
\label{Sec:GR08 Results}
We recall some of the main concepts for importance sampling with LDP, see \cite{GuasoniRobertson2008} and \cite{GlassermanEtAl1999} for further discussion.
We denote by $\bW^d_{T}$ the standard $d$-dimensional Wiener space of
continuous functions over the time interval $[0,T]$ which are zero at
time zero and in the one-dimensional case we simply write $\bW_T$
instead of $\bW^1_T$. This space is endowed with the topology of
uniform convergence and with the usual Wiener measure $\bP$, defined
on the completed filtration $\cF_{T}$, which makes the process $\mathbf{W}_{t}(x)=x_{t}$ with $x \in \bW_{T}^{d}$ a standard $d$-dimensional Brownian motion.

The goal is to estimate the expected value of some functional $\tilde{G}:\bW_{T}^{d} \rightarrow \bR_{+}$  continuous in the uniform topology ($\tilde{G}$ is explained later). For the change of measure, one considers a Girsanov transform where the allowed functions are from the Cameron-Martin space of absolutely continuous functions with square integrable derivative, i.e. (if $d=1$ we just write $\bH_T=\bH^1_T$)
\begin{align*}
\bH^d_{T} = \left \{ h:[0,T]\mapsto \bR^d: ~ h_{0}=0 \, , ~ h_{\cdot}=\int_0^\cdot \dot{h}_{t}dt\, , ~
 ~ \int_{0}^{T} |\dot{h}_{t}|^{2} \, \dd t < \infty\ \textrm{\ i.e. \ } \dot h_{t} \in L_{t}^2(\bR^d) \right \}.
\end{align*} 
For any deterministic drift $h \in \bH^d_{T}$, the stochastic exponential defines the Radon-Nikodym derivative for an equivalent measure $\bQ$ namely, ($W^\bP$ is a standard $\bP$-Brownian motion)
\begin{align}
\label{Eq:Deterministic Measure Change}
\frac{\dd \bQ}{\dd \bP}
=
\exp \Big( 
\int_{0}^{T}  \dot{h}_{t} \dd W_{t}^{\bP} - \frac{1}{2} \int_{0}^{T} (\dot{h}_{t})^{2} \, \dd t
\Big).
\end{align}
Under this new measure $\bQ$, the process $W_{\cdot}^{\bQ}= W_{\cdot}^{\bP}-h_{\cdot}$ is a standard $d$-dimensional $\bQ$-Brownian motion.


\subsubsection*{Standing assumptions}

We consider MV-SDEs with nonlinear interaction between the SDE and its
law. In this section we concentrate on one-dimensional SDEs of the form, 
\begin{align}
\label{Eq:More General MV-SDE}
\dd X_{t}= b(t,X_{t}, \mu_{t}) \dd t + \sigma \dd W_{t}, \qquad X_0=x_0.
\end{align}
Throughout this section the following assumptions area assumed to hold (similar to those in Section \ref{Sec:Representation}), for functions  $b:[0,T] \times \bR \times\cP_2(\bR) \to \bR$ and $\sigma>0$ constant.
\begin{assumption}
	\label{Ass:Drift Lipschitz Assumption}
	Assume that $b$ is Lipschitz in the sense that $\exists L>0$ such that $\forall t \in[0,T]$, $\forall x, x'\in \bR$ and $\forall \mu, \mu'\in \cP_2(\bR)$ we have that
	$$
	|b(t, x, \mu)-b(t, x',\mu')| \leq L(|x-x'| + W^{(2)}(\mu, \mu') ).
	$$
	Moreover, $\forall$ $x\in \bR$ and $\mu \in \cP_{2}(\bR)$, $b$ is continuous over the interval $[0,T]$.
\end{assumption}

\begin{assumption}
	\label{Ass:Drift Monotone Assumption}
	
	Assume $b$ satisfies the one-sided Lipschitz growth and local Lipschitz conditions in Assumption \ref{Ass:Monotone Assumption}. Further, $\forall$ $x\in \bR$ and $\mu \in \cP_{2}(\bR)$, let $b$ be continuous in time over the interval $[0,T]$.
\end{assumption}

In view of Section \ref{Sec:Representation}, either of these assumptions yield the existence of a unique strong solution to \eqref{Eq:More General MV-SDE}.
We further use the following assumption for the terminal function $G$. Note that this assumption is on $G$ as a function of the SDE, rather than the driving Brownian motion as is the case in \cite{GuasoniRobertson2008}.
\begin{assumption}
	\label{Ass:General Payoff Growth}
	The functional $G$ is non-negative, continuous and satisfies the following growth condition
	\begin{align*}
	\log(G(x)) \le 
	C_{1} + C_{2} \sup_{t \in [0,T] } |x_{t}|^{\alpha} \, ,
	\end{align*}
	for $x:[0,T]\mapsto \bR$ a continuous function starting at $x_0$ where $C_{1},~C_{2}$ are positive constants and $\alpha \in [1,2)$.
\end{assumption}
The notion of ``optimality'' for the measure change we use is so-called \emph{asymptotically optimal}, as defined in\footnote{A related but slightly weaker definition of optimality is used in \cite{GuasoniRobertson2008}.} \cite{GlassermanEtAl1999}. Following the approach of \cite{GlassermanEtAl1999}, we want to estimate $\bE[ \exp(\log(G(X))) ]$. Here we perform a measure change for the Brownian motion, so for ease of writing let us define $F(W) := \log(G(X(W)))$ and consider the more general problem of estimating,
\begin{align*}
\alpha( \epsilon) 
:=
\bE[ \exp( F(\sqrt{\epsilon}W)/\epsilon)], \qquad \text{for } \epsilon>0 .
\end{align*}
This is our original problem when $\epsilon=1$, and we can use Varadhan's lemma to understand this quantity as $\epsilon \rightarrow 0$, this is referred to as \emph{small noise asymptotics}. We now consider a general estimator for this quantity $\hat{\alpha}(\epsilon)$ (there is no requirement for $\hat{\alpha}$ to be based on a deterministic measure change). At this point we have no conditions on these estimators so we follow definition \cite{GlassermanEtAl1999}*{Definition 2.1}.
\begin{definition}
	\label{Defn:Asymptotically Unbiased}
	A family of estimators $\{ \hat{\alpha}(\epsilon)\}$ is said to be \emph{asymptotically relatively unbiased} if the following holds,
	\begin{align*}
	\frac{\bE[\hat{\alpha}(\epsilon)] - \alpha(\epsilon)}
	{\alpha(\epsilon)}
	\rightarrow 0
	\quad
	\text{as } \epsilon \rightarrow 0 \, .
	\end{align*}
\end{definition}

The above definition yields estimators that in some sense converge, but we are interested in comparing such estimators and for this we look at their second moment.
\begin{definition}
	\label{Defn:General Asymptotic Optimality}
	A family of relatively unbiased estimators $\{\hat{\alpha}_{0}(\epsilon)\}$ is said to be \emph{asymptotically optimal} if,
	\begin{align*}
	\limsup_{\epsilon \rightarrow 0} \epsilon \log \bE [\hat{\alpha}_{0}(\epsilon)^{2}]
	=
	\inf_{\{\hat{\alpha}(\epsilon)\}}
	\limsup_{\epsilon \rightarrow 0} \epsilon \log \bE [\hat{\alpha}(\epsilon)^{2}],
	\end{align*}
	where the infimum is over all asymptotically relatively unbiased estimators.
\end{definition}

One of the goals of this section will be obtaining conditions when measure changes of type \eqref{Eq:Deterministic Measure Change} are asymptotically optimal. As it turns out, using this definition it is not difficult to obtain a necessary and sufficient condition for asymptotic optimality, a similar argument is given in \cite{GlassermanEtAl1999}*{pg 126}. Let us consider some asymptotic unbiased estimator $\hat{\alpha}$, and define the difference $\Delta (\epsilon) := \bE[ \hat{\alpha}(\epsilon)] - \alpha(\epsilon)$, it is a straightforward consequence of Jensen's inequality and some rearranging,
\begin{align*}
\log( \bE[\hat{\alpha}(\epsilon)^{2}]) 
\ge
2 \log( \bE[\hat{\alpha}(\epsilon)])
=
2 \log( \bE[\alpha (\epsilon)])
+
O(\Delta(\epsilon)/ \alpha(\epsilon)) 
\xrightarrow{\epsilon \rightarrow 0}
2 \log( \bE[\alpha (\epsilon)])
\, .
\end{align*}
Thus we have a lower bound for an estimator, moreover, note that this implies the degenerate estimator $\hat{\alpha}(\epsilon) = \alpha(\epsilon)$ is asymptotically optimal, since $\alpha$ is not random. One can use Varadhan's lemma and Schilder's theorem (see Section \ref{Sec:Proofs Complete Measure Change}) since we are dealing with Brownian motion to obtain,
\begin{align}
\label{Eq:General Asymptotically Optimal Condition}
\limsup_{\epsilon \rightarrow 0} 2 \epsilon \log(\alpha(\epsilon))
=
\sup_{u \in \bH_{T}^{d}} \left\{
2\log(G(X(u))) - \int_{0}^{T} |\dot{u}_{t}|^{2} \dd t
\right\}.
\end{align}
Therefore any estimator which equals the RHS of this expression is asymptotically optimal. Depending on which algorithm we use this will be a slightly different expression but the argument to obtain the bound is the same.

\subsection{The decoupling algorithm}
\label{Sec:Measure Change for Decoupling}

We first consider the decoupling algorithm presented in Section \ref{sec:DecouplingIS}.
We build $\mu_{t}^{N}$, from an independent $N$-particle system which
is simulated under a  numerical scheme, and
then consider the following approximation of SDE\footnote{The measure, $\mu^{N}$ is a random measure, but is independent of the process $\overline{X}$ thus we have decoupled the SDE.} \eqref{Eq:More General MV-SDE},
\begin{align}
\label{Eq:Particle Approx General MV-SDE}
\dd \overline{X}_{t}= b(t,\overline{X}_{t}, \mu_{t}^{N}) \dd t + \sigma\dd W_{t}, \qquad X_0=x_0 \, .
\end{align}
In order to distinguish the current SDE from the previous particle approximation we introduce a so-called copy space (see for example \cite{BuckdahnEtAl2017}) $(\tilde{\Omega}, \tilde{\cF} , (\tilde{\cF}_{t})_{t \ge 0}, \tilde{\bP})$ (with the usual conditions and $\tilde{\cF}_{t}$ is the augmented filtration over the $N$-dimensional Brownian motion). The $N$-system SDEs used to approximate this measure is then defined on this space, hence \eqref{Eq:Particle Approx General MV-SDE} is defined on the product space $(\Omega, \cF, \bP) \otimes (\tilde{\Omega}, \tilde{\cF} , \tilde{\bP})$.

Our aim is now to minimize over $h\in \mathbb H_T$ the variance conditional on the trajectory
of $\mu^N$:
$$
\bE_{\bP \otimes \tilde{\bP}}\big[\, G(\overline X_{T})^{2} \mathcal E_T^{-1}\big|\tilde {\mathcal   F}_{T}\big],
\quad \dd \mathcal E_t = \dot h_t \mathcal E_t \dd W_{t}^{\bP},\quad
\mathcal E_0 = 1,
$$
and we make  use of small noise asymptotics in order to write this variance in a ``LDP'' tractable form, hence we define, for $h\in\bH_T$
\begin{align}
\label{Eq:Small Noise for General MV-SDE}
L(h; \mu^{N})
:=
\limsup_{\epsilon \rightarrow 0}
\epsilon \log \bE_{\bP\otimes \tilde{\bP}} \left[ 
\exp\left( \frac{1}{\epsilon}\Big(2 \log(\overline{G}(\sqrt{\epsilon} W)) - \int_{0}^{T} \sqrt{\epsilon} \dot{h}_{t} \dd W_{t} + \frac{1}{2} \int_{0}^{T} \dot{h}_{t}^{2} \dd t \Big) \right)
\Big |
 \tilde{\cF}_{T}
\right] \, ,
\end{align}
where $\overline{G}(W) := G(\overline{X}(W))$. One should also keep in mind that $\overline{G}$ also depends on $\mu^{N}$, however, we suppress this notation for ease of presentation.

\begin{remark}
	In \eqref{Eq:Small Noise for General MV-SDE}, we have a conditional expectation, thus $L(h; \mu^{N})$ is technically a random variable in $\tilde{\Omega}$. This is not typically the case when using Varadhan's lemma, however, because the random variable is independent of the Brownian motion and $\overline{G}$ is still $\tilde{\bP}$-a.s. continuous w.r.t. the Brownian motion (Section \ref{Sec:Proofs Decoupled}), upon checking the moment condition, we are still able to use Varadhan's lemma, $\tilde{\bP}$-a.s..
\end{remark}


\begin{theorem}
	\label{Thm:Decoupled}
	Let Assumptions \ref{Ass:General Payoff Growth} and
        \ref{Ass:Drift Monotone Assumption} hold and fix
        $\tilde{\omega} \in \tilde{\Omega}$ (and thus $\mu^{N}$). Furthermore assume  that
        there exists $u \in \bH_{T}$ such that $\overline{G}(u) >0$. Then the following statements hold:
	\begin{itemize}
		\item[i.] Let $h\in\bH_T$ such that $\dot{h}$ is of finite
                  variation. Then Varadhan's lemma holds for the small noise asymptotics, namely we can rewrite \eqref{Eq:Small Noise for General MV-SDE} as,
		\begin{align}
		\label{Eq:Decoupled L in sup form}
		L(h; \mu^{N})=
		\sup_{u \in \bH_{T}} \left\{
		2 \log (\overline{G}(u))
		-
		\int_{0}^{T} \dot h_{t} \dot u_{t} \dd t
		+
		\frac{1}{2} \int_{0}^{T} \dot h_{t}^{2} \dd t 
		-
		\frac{1}{2} \int_{0}^{T} \dot u_{t}^{2} \dd t
		\right\} 
		~ ~ ~ \tilde{\bP}\text{-a.s.} \, .
		\end{align}
		\item[ii.] There exists an $h^* \in \bH_{T}$ which minimizes \eqref{Eq:Decoupled
                    L in sup form}. 
		\item[iii.] Consider a simplified optimization problem
		\begin{align}
		\label{Eq:General Calculus of Variation problem}
		\sup_{u \in \bH_{T}} \left\{
		2 \log(\overline{G}(u)) - 
		\int_{0}^{T} \dot u_t^2  \dd t
		\right\} \, .
		\end{align}
 There exists  a maximizer $h^{**}$ for this problem. If 
		\begin{align}
		\label{Eq:General Asymptotic Optimality check}
		L({h}^{**} ; \mu^{N})=2 \log(\overline{G}({h}^{**})) - 
		\int_{0}^{T} (\dot {h}^{**}_{t} )^{2} \dd t
		\, ,
		\end{align}
	\end{itemize}
then $h^{**}$ defines an asymptotically optimal measure change and is the unique maximizer of \eqref{Eq:General Calculus of Variation problem}.
\end{theorem} 
All of these results are $\tilde{\bP}$-a.s. since the particle system yields a random measure from $\tilde{\Omega}$.
The proof of this theorem requires several auxiliary results which we
defer to Section \ref{Sec:Proofs Decoupled}. One should also note that
the requirement for $\overline{G}>0$ for some $u$ is not restrictive, it is purely
there for technical reasons since one cannot have a maximiser if
$\log(\overline{G}(u))= - \infty$ for all $u \in \bH_{T}$. The
assumption that $\dot h$ has finite variation is necessary to
establish the continuity of the functional in Varadhan's lemma.

\begin{remark}
	[Concavity of $\log(\overline{G})$ and asymptotic optimality]
Consider the problem of minimizing \eqref{Eq:Decoupled L in sup form}
and assume that one can interchange the inf and the sup. Then,
\begin{align*}
\inf_{h \in \bH_{T}}  L(h; \mu^{N})& =
		\sup_{u \in \bH_{T}}\inf_{h \in \bH_{T}}  \left\{
		2 \log (\overline{G}(u))
		-
		\int_{0}^{T} \dot h_{t} \dot u_{t} \dd t
		+
		\frac{1}{2} \int_{0}^{T} \dot h_{t}^{2} \dd t 
		-
		\frac{1}{2} \int_{0}^{T} \dot u_{t}^{2} \dd t
		\right\} \\
&=
		\sup_{u \in \bH_{T}} \left\{
		2 \log (\overline{G}(u))
		-
		\int_{0}^{T} \dot u_{t}^{2} \dd t
		\right\} 
\end{align*}
because the inner problem is solved by $h = u$. Therefore, a
sufficient condition for an asymptotically optimal measure change of type \eqref{Eq:Deterministic Measure Change} is the
exchangeability of inf and sup above. Since $L$ is a convex
        function in $h$, and the integral terms in \eqref{Eq:Decoupled
          L in sup form} are concave in $u$, a sufficient condition
        for such exchangeability is that
        $\log(\overline{G})$ is concave. Indeed, in the case of convex-concave functions we can invoke the minimax principle to swap infimum and supremum, see \cite{EkelandTemam1999}*{pg. 175} for  example.

        In \cite{GuasoniRobertson2008}, the process $X$ was a
        geometric Brownian Motion and the authors were able to explicitly
        link the concavity of $\log(\overline{G})$ with the properties
        of the function $G$. Here the dependence of $\overline{G}$ on
        the Brownian motion is more complex, and it appears to be difficult to
        check concavity.
	Hence, in general one has to check numerically whether
        \eqref{Eq:General Asymptotic Optimality check} holds. However,
        even if \eqref{Eq:General Asymptotic Optimality check} fails,
        one can still  use $h^{**}$ to construct a candidate
        importance sampling measure if this is justified by superior
        numerical performance.

\end{remark}

\subsection{The complete measure change algorithm}

Here we focus on the algorithm discussed in Section
\ref{Sec:Complete Measure Change}. Recall that we are interested in evaluating,
$\bE_{\bP}[G(X)]$. We now change the measure to $\bQ$ and calculate the variance,
\begin{align*}
	\text{Var}_{\bQ} \Big[ G(X) \frac{\dd \bP}{\dd \bQ} \Big]
	=
	\bE_{\bP}\Big[ G(X)^{2} \frac{\dd \bP}{\dd \bQ} \Big]
	-
	\bE_{\bP}\Big[ G(X) \Big]^{2} \, .
\end{align*}
Minimising the variance is equivalent to minimize the first term in
the RHS. As a first step to constructing a tractable proxy for this variance we consider a particle approximation of $X$:
\begin{align}
	\label{Eq:Form of MV-SDE}
	\dd X_{t}^{i,N}
	&=
	b\left( t, X_{t}^{i,N}, \frac{1}{N} \sum_{j=1}^{N} \delta_{X_{t}^{j,N}} \right) \dd t + \sigma \dd W_{t}^{i, \bP} \, , 
	\quad 
	X_{0}^{i,N} = x_0 \, ,
	\\
	\dd \mathcal E^{i}_t & = \dot h_t \mathcal E^{i}_t \dd
	W^{i, \bP}_t,\quad \mathcal E^{i}_0 = 1,
\end{align}
where $W^{i,\bP}$ denotes the driving $\bP$-Brownian motion of particle $i$, and
all $W^{i,\bP}$s are independent of each other. We approximate
$\bE_{\bP}[G^{2}(X) (\cE_T)^{-1}]$ with $\mathbb E_\bP[
G^2(X^{i,N})(\mathcal E_T^{i,N})^{-1}]$. Since $\mathcal
E^{i}=\mathcal E^{i,N}$ (due to the absence of cross dependency), one
can equivalently minimize 
\begin{align}
	\ \mathbb E_\bP\big[\ G^2(X^{i,N})(\mathcal E_T^{i})^{-1}\ \big]
	\, ,\label{prelimit} \ 
	\textrm{ over all $h\in \mathbb H_T$.}
\end{align}
In order to use the LDP theory to minimize \eqref{prelimit}, we define
$\tilde{G}$ as the functional dependent on the underlying $\bP$-Brownian
motions, i.e., for all $i \in \{1, \dots, N\}$,  $\tilde{G}_{i}
:\bW^N_T\mapsto \bR$, where, $\tilde{G}_{i}(W^{1}, \dots, W^{N}):=
G(X^{i,N}(W^{1}, \dots, W^{N}))$. The corresponding small noise
asymptotics takes the following form:
\begin{align}
	\label{Eq:Small Noise for Particles}
	\notag
	\bar{L}(h):=
	\limsup_{\epsilon \rightarrow 0} 
	\epsilon \log \Bigg(
	\bE_{\bP} \Bigg[
	\exp \Bigg(
	\frac{1}{\epsilon} \Big( 
	2 \log & \big(\tilde{G}_i\big(\sqrt{\epsilon}W^{1}, \dots, \sqrt{\epsilon}W^{N}\big)\big)
	\\
	&-
	\int_{0}^{T} \sqrt{\epsilon}  \dot h_{t} \dd W^i_{t} 
	+
	\frac{1}{2} \int_{0}^{T} (\dot h_{t})^{2} \dd t
	\Big)
	\Bigg)
	\Bigg]
	\Bigg) \, ,
	\quad
	h \in \bH_{T}
\end{align}
where we remark that the value of this expression does not depend on
the choice of $i$. 
 We then obtain the following result for $\bar{L}$ (compare with Theorem \ref{Thm:Decoupled}).

\begin{theorem}
	\label{Thm:Complete Measure Change}
	Fix $N\in \bN$ and let Assumptions \ref{Ass:General Payoff
          Growth} and \ref{Ass:Drift Lipschitz Assumption}
        hold. Assume that there exists $(u^{1}, \hat{u}) \in \bH_{T}^{2}$ such that $\tilde{G}_{1}(u^{1}, \hat{u}, \dots, \hat{u})>0$. Then the following statements hold
	\begin{itemize}
		\item[i.] Let $h\in \bH_T$ such that $\dot h$ is of
                  finite variation. Then Varadhan's lemma holds for the small noise asymptotics and we can rewrite \eqref{Eq:Small Noise for Particles} as
	\begin{align}
		\label{Eq:Complete L in sup form}
		\bar{L}(h)=
		\sup_{u \in \bH^N_{T}} \left\{
		2 \log (\tilde{G}_{1}(u^{1}, \dots, u^{N}))
		-
		\int_{0}^{T} \dot h_{t} \dot u_{t}^{1} \dd t
		+
		\frac{1}{2} \int_{0}^{T} (\dot h_{t})^{2} \dd t 
		-
		\frac{1}{2} \int_{0}^{T} |\dot u_{t}|^{2} \dd t
		\right\} \, ,
	\end{align}
	\item[ii.] There exists an $h^* \in \bH_{T}$
          which minimizes \eqref{Eq:Complete L in sup form}. 
	\item[iii.] Consider a simplified optimization problem
	\begin{align}
	\label{Eq:Complete L minimiser}
	\sup_{u^{1} \in \bH_{T}, \hat{u} \in \bH_{T}} \left\{
	2 \log(\tilde{G}_{1}(u^{1}, \hat{u}, \dots, \hat{u}) ) 
	- 
	\int_{0}^{T} (\dot u_{t}^{1})^{2} \dd t
	-
	\frac{N-1}{2} \int_{0}^{T} \dot{\hat{u}}_{t}^{2} \dd t
	\right\} \, .
	\end{align}
There exists a maximizer $(h^{**},u^{**})$ for this problem. If
		\begin{align}
	\label{Eq:Complete Asymptotic Optimality check}
	\bar{L}({h}^{**})
	=	
	2 \log \big(\tilde{G}_{1}({h}^{**}, {u}^{**}, \dots, {u}^{**}) \big) 
	- 
	\int_{0}^{T} (\dot {h}^{**}_{t})^{2} \dd t
	-
	\frac{N-1}{2} \int_{0}^{T} (\dot{u}^{**}_{t})^{2} \dd t
	 \, .
	\end{align}
then $h^{**}$ is asymptotically optimal and is the unique maximizer of
\eqref{Eq:Complete L minimiser}, where we have taken $i=1$ without loss of generality.
	\end{itemize}
\end{theorem} 
The proof of this theorem is deferred to Section \ref{Sec:Proofs
  Complete Measure Change}. Similarly to the previous discussion if
$\log(\tilde{G}_{1})$ is a concave function in $u \in \bH_{T}^{N}$,
then we know that \eqref{Eq:Complete Asymptotic Optimality check} holds (this is discussed at the end of Section \ref{Sec:Proofs Complete Measure Change}). However, in general \eqref{Eq:Complete Asymptotic Optimality check} is difficult to check since, even with $h^{*}$ fixed, $\bar{L}$ is still an $N$-dimensional optimisation problem, since \eqref{Eq:Complete L in sup form} is supremum over $u \in \bH_{T}^{N}$.

There is also a difficulty in quantifying how the measure change affects the propagation of chaos error i.e. a measure change that is good for the statistical error may be damaging to the propagation of chaos error. We discuss this point further in Section \ref{sec:Numerics}.


\subsection{Computing the optimal measure change}

The exponential form of the SDEs (the log-normal class) considered in
\cite{GuasoniRobertson2008} and \cite{Robertson2010} allows the
maximisation to be written in the form of an Euler-Lagrange equation
(calculus of variations approach). Due to the more general
coefficients here, we obtain a more complex interaction between the
Brownian motion and the value of the SDE. Consequently we need to look
towards the more general theory of optimal control to calculate the
change of measure\footnote{Even though we are initially dealing with
  SDEs, in the large deviations asymptotics, the trajectory of the
  Brownian motion becomes a deterministic control.}. Deterministic  optimal control is a large subject area and one can consult \cite{FlemingRishel1975} or \cite{YongZhou1999} for example. We recall that we are working under the $\bP$-measure.

One of the most important results from optimal control is Pontryagin's maximum principle. Roughly speaking, Pontryagin's maximum principle gives a set of differential equations that the optimal control must satisfy. 
Let us recall the main ideas following \cite{YongZhou1999}*{p.102}. We
start with the controlled dynamical system $x(t)$ which takes the
following form:
\begin{align}
\label{Eq:General Control System}
\begin{cases}
\dot{x}(t)=b(t,x(t),u(t)), \quad \text{a.e.} ~ t \in [0,T]
\\
x(0)=x_{0} \, ,
\end{cases}
\end{align}
where $u$ is our ``control'', which is defined in a metric space $(U,d)$ and associated to this we have a \emph{cost functional}
\begin{align}
\label{Eq:General Cost Functional} 
J(u(\cdot)) =
\int_{0}^{T} f(t, x(t), u(t)) \dd t + h(x(T)) \, ,
\end{align}
$f$ is typically referred to as the \emph{running cost} and $h$ the \emph{terminal cost}. We then have the following assumption.
\begin{assumption}
	\label{Ass:Optimal Control Assumption}
	For ease of writing we denote by $\varphi(t,x,u)$ to be any of the functions $b(t,x,u), ~ f(t,x,u)$ or $h(x)$. We then assume the following,
	\begin{itemize}
		\item $(U,d)$ is a separable metric space and $T>0$.
		
		\item The maps $b:[0,T] \times \bR^{n} \times U \rightarrow \bR^{n}$, $f:[0,T] \times \bR^{n} \times U \rightarrow \bR$ and $h: \bR^{n} \rightarrow \bR$ are measurable and there exists a constant $L>0$ and a modulus of continuity $\eta : [0, \infty) \rightarrow [0, \infty)$ such that,
		\begin{align*}
		\begin{cases}
		| \varphi(t,x,u)-\varphi(t, \hat{x}, \hat{u})| \le 
		L|x- \hat{x}| + \eta(d(u, \hat{u}))  
		\quad 
		&\forall t \in [0,T] ~ x,\hat{x} \in \bR^{n}, ~ u,\hat{u} \in U \, ,
		\\
		|\varphi(t,0,u)| \le L 
		& \forall (t,u) \in [0,T] \times U \, .
		\end{cases}
		\end{align*}
		
		\item The maps $b, ~ f$ and $h$ are $C^{1}$ in $x$ and there exists a modulus of continuity $\eta : [0, \infty) \rightarrow [0, \infty)$ such that,
		\begin{align*}
		| \partial_{x} \varphi(t,x,u) - \partial_{x} \varphi(t, \hat{x}, \hat{u})|
		\le
		\eta \Big(|x - \hat{x}| + d(u, \hat{u}) \Big)
		\quad 
		\forall t \in [0,T] ~ x,\hat{x} \in \bR^{n}, ~ u,\hat{u} \in U \, .
		\end{align*}
	\end{itemize}
\end{assumption}
As discussed in \cite{YongZhou1999}*{p.102}, Assumption
\ref{Ass:Optimal Control Assumption} implies that \eqref{Eq:General Control System} admits a unique solution and \eqref{Eq:General Cost Functional} is well defined. Let us denote by $\cU [0,T] := \{ u(\cdot): [0,T] \rightarrow U ~|~ u \text{ is measurable}\}$, then optimal control problem is to find $u^{*} \in \cU[0,T]$ that satisfies,
\begin{align}
\label{Eq:General Optimal Control Problem}
J(u^{*}) = \inf_{u \in \cU[0,T]} J(u) \, .
\end{align}
Such $u^{*}$ is referred to as an \emph{optimal control}, and the corresponding $x^{*}(\cdot):=x(\cdot ; u^{*})$ the \emph{optimal state trajectory}.
We can then state the deterministic version of Pontryagin's maximum principle as \cite{YongZhou1999}*{p.103}.
\begin{theorem}
	\label{Thm:Pontryagin}
	[Pontryagin's Maximum Principle]
	Let Assumption \ref{Ass:Optimal Control Assumption} hold and
        let $(x^{*}, u^{*})$ be the optimal pair to \eqref{Eq:General
          Optimal Control Problem}. Then, there exists a function $p: [0,T] \rightarrow \bR^{n}$ satisfying the following,
	\begin{align}
	\label{Eq:General Adjoint Equation}
	\begin{cases}
	 \dot{p}(t) = - \partial_{x} b(t, x^{*}(t), u^{*}(t))^{\intercal} p(t) + \partial_{x} f(t, x^{*}(t), u^{*}(t)), 
	 \quad \text{a.e.} ~ t \in [0,T]
	 \\
	 p(T)= - \partial_{x} h(x^{*}(T)) \, ,
	\end{cases}
	\end{align}
	and
	\begin{align*}
	H(t, x^{*}(t), u^{*}(t), p(t)) 
	=
	 \max_{u \in U} \{ H(t, x^{*}(t), u, p(t)) \}
	 \quad \text{a.e.} ~ t \in [0,T] \, ,
	\end{align*}
	where $H(t,x,u,p):=	\langle p, b(t,x,u)\rangle - f(t,x,u)$ for any $(t,x,u,p) \in [0,T] \times \bR^{n} \times U \times \bR^{n}$.
\end{theorem}

Typically $p$ is referred to as the \emph{adjoint function} and \eqref{Eq:General Adjoint Equation} the \emph{adjoint equation}, and the function $H$ is called the \emph{Hamiltonian}.

\begin{remark}
	[An alternative approach]
	The maximum principle is not the only way one can use to solve
        this problem. An alternative is by solving the so-called
        Hamilton-Jacobi-Bellman (HJB) equation. This approach is
        typically more difficult since the HJB is a \emph{partial
          differential equation}. 
\end{remark}

\textbf{Maximum principle for Theorems \ref{Thm:Decoupled} and \ref{Thm:Complete Measure Change}.}
The maximum principle allows to translate the simplified optimization
problems of Theorems \ref{Thm:Decoupled} and \ref{Thm:Complete Measure
  Change} into boundary value problems for ODE. One can observe that
we are actually interested in $\dot{u}$ rather than $u$, that is, in
the decoupled case we can write the controlled dynamics as
\begin{align*}
X_{t}(\dot{u})
= x_{0} + \int_{0}^{t} b(s, X_{s}(\dot{u}), \mu_{s}^{N}) \dd s + \int_{0}^{t} \sigma \dot{u}_{s} \dd s \, . 
\end{align*}
The theory above is for infimum while we are interested in supremum, therefore we use the fact that $\sup \{f\} = - \inf \{-f\}$. 

$\triangleright$ \emph{For the decoupling algorithm} Theorem \ref{Thm:Pontryagin} yields the following equations for the adjoint function and trajectory under optimal control $\dot{u}^{*}$ (for a given $\mu^{N}$),
\begin{align}
\label{Eq:Decoupled Maximum Principle}
\text{(Decoupled)} 
\quad
\begin{cases}
\dot{p}_{t} = - \partial_{x} b(t, X_{t}(\dot{u}^{*}), \mu_{t}^{N}) p_{t}
\, ,
\qquad 
&
p_{T}= \dfrac{2 G'(X(\dot{u}^{*}))}{G(X(\dot{u}^{*}))} \, ,
\\
\dot{X}_{t} = b(t, X_{t}, \mu_{t}^{N}) + \frac{1}{2} \sigma^{2} p_{t} \, ,
&
X_{0}=x_{0} \, ,
\end{cases}
\end{align}
that is, the optimal control is related to $p$ through, $\dot{u}^{*}_{t}= \frac{1}{2} \sigma p_{t}$.

$\triangleright$ \emph{For the complete measure change algorithm} the
argument is similar argument to the above one, but here we also need to deal with the measure term. Noting that we have two controls to optimise over (recall Theorem \ref{Thm:Complete Measure Change}) we obtain more complex expressions. Theorem \ref{Thm:Pontryagin} yields the following system of ODEs,
\begin{align}
\label{Eq:Complete Maximum Principle}
\text{(Complete)} 
 ~
\begin{cases}
\dot{p}_{t}^{1} = - \partial_{X^{1}} b(t, X_{t}^{1}, \frac{1}{N}\delta_{X_{t}^{1}} + \frac{N-1}{N}\delta_{\hat{X}_{t}} ) p_{t}^{1}
- \partial_{X^{1}} b(t, \hat{X}_{t}, \frac{1}{N}\delta_{X_{t}^{1}} + \frac{N-1}{N}\delta_{\hat{X}_{t}} ) p_{t}^{2}
\, ,
~
& 
p_{T}^{1}= \frac{2 G'(X^{1})}{G(X^{1})} \, ,
\\
\dot{p}_{t}^{2} = - \partial_{\hat{X}} b(t, X_{t}^{1}, \frac{1}{N}\delta_{X_{t}^{1}} + \frac{N-1}{N}\delta_{\hat{X}_{t}} ) p_{t}^{1}
- \partial_{\hat{X}} b(t, \hat{X}_{t}, \frac{1}{N}\delta_{X_{t}^{1}} + \frac{N-1}{N}\delta_{\hat{X}_{t}} ) p_{t}^{2}
\, ,
~
& 
p_{T}^{2}= 0 \, ,
\\
\dot{X}_{t}^{1} = b(t, X_{t}^{1}, \frac{1}{N}\delta_{X_{t}^{1}} + \frac{N-1}{N}\delta_{\hat{X}_{t}}) + \frac{1}{2} \sigma^{2} p_{t}^{1} \, ,
&
X_{0}^{1}=x_{0} \, ,
\\
\dot{\hat{X}}_{t} = b(t, \hat{X}_{t}, \frac{1}{N}\delta_{X_{t}^{1}} + \frac{N-1}{N}\delta_{\hat{X}_{t}}) +  \frac{1}{2(N-1)} \sigma^{2} p_{t}^{2} \, ,
&
\hat{X}_{0}=x_{0} \, ,
\end{cases}
\end{align}
similarly we obtained, $\dot{u}^{*}_{t}= \frac{1}{2} \sigma p_{t}^{1}$ and $\dot{\hat{u}}^{*}_{t}= 0$ as the optimal controls. From Theorem \ref{Thm:Complete Measure Change} we obtain the measure change as $\dot{h} = \dot{u}$.

The difference between \eqref{Eq:Decoupled Maximum Principle} and \eqref{Eq:Complete Maximum Principle} comes from the fact that for the complete measure change we have a higher dimensional problem. That is, we have two controls and two ``SDEs'' thus we have more terms to optimise. 
Recall, when one wishes to assess asymptotic optimality, \eqref{Eq:Complete L in sup form} is still an $N$-dimensional problem.

\begin{remark}
	[Accuracy of Change of Measure] 
	In \cite{GuasoniRobertson2008}, they were able to obtain
        explicit solutions in certain situations, but here, due to the increase
        in complexity,  we expect this to rarely be the case. We
        therefore need to set reasonable tolerances in checking
        whether asymptotic optimality holds.
\end{remark}

%
%
%
%

%
%
%
%

\section{Example: Kuramoto model}
\label{sec:Numerics}

The Kuramoto model is a special case of a so-called system of coupled oscillators. Such models are of particular interest in physics and are used to study many different phenomena such as active rotator systems, charge density waves and complex biological systems amongst other things, see \cite{KosturEtAl2002} for further details.
The corresponding SDE for the Kuramoto model is
\begin{align*}
\dd X_{t} = \left( K \int_{\bR}\sin (y - X_{t}) \mu_{t, \bP}^{X}(\dd y) -\sin(X_{t}) \right) \dd t + \sigma \dd W_{t}^\bP \, ,
~ ~ ~ 
t \in [0,T] , ~ ~ X_{0}=x_{0} \, ,
\end{align*}
where $K$ is the coupling strength and $\sigma$ has the physical interpretation of the temperature in the system.
We consider a terminal condition $G(x)= a \exp (bx)$ (satisfying Assumption \ref{Ass:General Payoff Growth}). Our goal is to obtain the asymptotically optimal change of measure that improves the estimation of $\bE_\bP[ G(\bar{X}_{T})]$. 

One can see that such a model easily satisfies the assumptions required in the paper. Let us now apply the theory from the previous section to calculate the optimal change of measure. We should point out here that we do not have the concavity required for asymptotic optimality to hold automatically, therefore we need to check this condition.

By our previous discussion, to apply the decoupling algorithm here we would generate a set of $N$ weakly interacting SDEs which we denote by $Y^{i,N}$ and approximate the original SDE by,
\begin{align*}
\dd \bar{X}_{t} = \left( \frac{K}{N} \sum_{i=1}^{N} \sin (Y_{t}^{i,N} - \bar{X}_{t}) -\sin(\bar{X}_{t}) \right) \dd t + \sigma \dd W_{t}^\bP \, ,
~ ~ ~ 
t \in [0,T] , ~ ~ \bar{X}_{0}=x_{0} \, .
\end{align*}
Let us now apply the theory from the previous section to calculate the optimal change of measure.
Our optimal control argument implies solving $\tilde \bP$-a.s.
\begin{align*}
\text{(Decoupled)} 
\quad
	\begin{cases}
		\dot{p}_{t} = \left( \frac{K}{N} \sum_{i=1}^{N} \cos(Y_{t}^{i,N} - X_{t}) + \cos(X_{t}) \right)  p_{t}
		\, ,
		\qquad 
		&
		p_{T}= 2b \, ,
		\\
		\dot{X}_{t} = \left( \frac{K}{N} \sum_{i=1}^{N} \sin (Y_{t}^{i,N} - X_{t}) -\sin(X_{t}) \right) + \frac{1}{2} \sigma^{2} p_{t} \, ,
		&
		X_{0}=x_{0} \, .
	\end{cases}
\end{align*}

The complete measure change algorithm yields the following system,
\begin{align*}
\text{(Complete)} 
~
\begin{cases}
\dot{p}_{t}^{1} =K \big(\frac{N-1}{N}\cos(\hat{X}_{t} - X_{t}^{1}) - \cos(X_{t}^{1}) \big) p_{t}^{1}
- \frac{K}{N}\cos(X_{t}^{1}- \hat{X}_{t}) p_{t}^{2}
\, ,
~
& 
p_{T}^{1}= 2b \, ,
\\
\dot{p}_{t}^{2} = - K \frac{N-1}{N} \cos(\hat{X}_{t} - X_{t}^{1}) p_{t}^{1}
+ K \big( \frac{1}{N} \cos(X_{t}^{1}- \hat{X}_{t}) + \cos(\hat{X}_{t}) \big) p_{t}^{2}
\, ,
~
& 
p_{T}^{2}= 0 \, ,
\\
\dot{X}_{t}^{1} = K \big( \frac{N-1}{N} \sin (\hat{X}_{t} - X_{t}^{1}) - \sin(X_{t}^{1}) \big) + \frac{1}{2} \sigma^{2} p_{t}^{1} \, ,
&
X_{0}^{1}=x_{0} \, ,
\\
\dot{\hat{X}}_{t} = K \big( \frac{1}{N} \sin(X_{t}^{1} - \hat{X}_{t}) - \sin( \hat{X}_{t}) \big) +  \frac{1}{2(N-1)} \sigma^{2} p_{t}^{2} \, ,
&
\hat{X}_{0}=x_{0} \, ,
\end{cases}
\end{align*}

To show the numerical advantages one can achieve by using importance
sampling we consider how the time taken and the estimate given by the
algorithms change with the number of particles $N$.


 For this example we use, $T=1$, $\bar{X}_{0}=0$, $K=1$, $\sigma=0.3$,
 $a=0.5$ and $b=10$. For the numerics we use an Euler scheme with step size of $\Delta t=0.02$. The systems of equations are solved using MATLAB's \verb|bvp4c| function. For the importance sampling, we use the particle positions from the first Monte Carlo simulation as the empirical law.

\begin{table}[!ht] 
	\centering
	\small
	\begin{tabular}{ c | c |  c | c | c | c | c | c | c | c}
		~ & \multicolumn{3}{c}{Monte Carlo} \vline &  \multicolumn{3}{c}{Decoupled} \vline &  \multicolumn{3}{c}{Complete} \\ \hline
		N & Payoff & Error & Time & Payoff & Error & Time & Payoff & Error & Time
		\\ \hline
		
		$1 \times 10^3$ & 1.5066 & 0.1490 & 3 & 1.5729 & 0.0028 & 9 
		& 1.5419 & 0.0024 & 3 \\ 
		
		$5\times 10^3$ & 1.5895 & 0.0626 & 27 & 1.5840 & 0.0013 & 54
		& 1.5710 & 0.0013 & 28 \\
		
		$1 \times  10^4$ & 1.6813 & 0.0693 & 76 & 1.5728 & 0.0009 & 153
		& 1.5860 & 0.0009 & 75 \\
		
		$5\times 10^4$ & 1.5899 & 0.0200 & 1 025 & 1.5820 & 0.0004 & 2 052
		& 1.5738 & 0.0004 & 1 062  \\
		
		$1 \times  10^5$ & 1.5807 & 0.0176 & 3 433 & 1.5731 & 0.0003 & 6 935 
		& 1.5882 & 0.0003 & 3 644 \\
		\hline
	\end{tabular}
	\caption{Results from standard Monte Carlo and the importance sampling algorithms. Time is measured in seconds and error refers to square root of the variance.}
	\label{Table:Results from Algorithms}
\end{table}
We recall that the decoupling importance sampling requires two runs, here we use the same $N$ for both of these.
The first note one can make is how the time scales when increasing the
number of particles, namely one can truly observe the $N^{2}$
complexity. 
As expected the decoupling algorithm takes approximately twice as long as the standard Monte Carlo
(computing the change of measure is not time consuming). Following
this point we also observe that the complete measure change has
roughly the same computational complexity as standard Monte Carlo. The
other key point is the reduction in variance (standard error) one
obtains with importance sampling. For this example we see that both
importance sampling schemes reduce the variance by several orders of
magnitude. Further, if one is interested in the decoupling algorithm
it may be more efficient to take less simulations in the second
importance sampled run. Finally, we checked the asymptotic optimality (for the decoupling) numerically and there is only a small difference between the two sides in \eqref{Eq:Complete Asymptotic Optimality check}, we therefore believe we are close to the optimal.
Table \ref{Table:Results from Algorithms} does show that the use of importance sampling in MV-SDEs is both viable and worthwhile.

$\triangleright$ \emph{Estimating the propagation of chaos error.}
As was mentioned in the introduction, theoretically the statistical
error and the propagation of chaos error converge to zero at the same
rate. We now use this example to show that the statistical error dominates. Since the Euler scheme is the same in all examples we can neglect the bias caused by that. We can then decompose the error as
\begin{align*}
\frac{1}{N} \sum_{i=1}^{N} G(\bar{X}^{i,N}) - \bE_{\bP}[G(\bar{X}^{1})]
=
\frac{1}{N} \sum_{i=1}^{N} G(\bar{X}^{i,N}) - \bE_{\bP}[G(\bar{X}^{1,N})]
+
\bE_{\bP}[G(\bar{X}^{1,N})] - \bE_{\bP}[G(\bar{X}^{1})] \, .
\end{align*}
The first difference on the RHS is the statistical error, and the
second one is the propagation of chaos error. It is then clear that if
one considers $M$ realisations of $\frac{1}{N} \sum_{i=1}^{N}
G(\bar{X}^{i,N})$ and takes the average this approximates
$\bE_{\bP}[G(\bar{X}^{1,N})]$ but does not change the propagation of
chaos error. Hence for large $M$ the error reduces to the propagation of chaos error.
To show the propagation of chaos error is negligible compared to the
statistical error here, we repeat the simulation for $N=5 \times 10^{3}$ particles, $M=10^{3}$ times and we obtain an average terminal value of $1.5772$ (with an average standard error of $0.06533$, which agrees with the result in Table \ref{Table:Results from Algorithms}). Comparing this to the $10^{5}$ decoupled entry (which has almost no statistical error) in Table \ref{Table:Results from Algorithms}, we can conclude the propagation of chaos error at least an order of magnitude smaller than the statistical error.

\subsubsection*{Another example: a terminal condition function with steep slope}
Let us consider the terminal condition $G(x)= \big(\tanh(a(x-b))+1
\big)/2$, for $a$ large ($G$ can be understood as a mollified
indicator function). Then $\bE_\bP[G(X_{T})] \approx \bP(X_{T} \ge
b)$. We take the same set up as before but with $a=15$ and $b=1$ and note that the terminal condition for adjoint takes the form,
\begin{align*}
p_{T} = 2a \Big(1- \tanh \big(a \big( X_{T}(\dot{u}^{*}) -b \big) \big) \Big) \, .
\end{align*}
We obtain the following table (we omit the times here since they are similar).

\begin{table}[!ht] 
	\centering
	\small
	\begin{tabular}{ c | c | c | c | c | c | c}
		~ & \multicolumn{2}{c}{Monte Carlo} \vline &  \multicolumn{2}{c}{Decoupled} \vline &  \multicolumn{2}{c}{Complete} \\ \hline
		N & Payoff ($10^{-9}$) & Error ($10^{-9}$) & Payoff ($10^{-9}$) & Error ($10^{-9}$) & Payoff ($10^{-9}$) & Error ($10^{-9}$) \\ \hline
		
		$1 \times 10^3$ & 1.015 & 0.671 & 3.864 & 0.0250 & 8.456 & 0.101 \\ 
		
		$5\times 10^3$ & 1.093 & 0.752  & 3.952 & 0.0112 & 5.564 & 0.0185  \\
		
		$1 \times  10^4$ & 8.829 & 7.071  & 3.910 & 0.0077 & 32.956 & 0.1520 \\
		
		$5\times 10^4$ & 1.106 & 0.271 & 3.970 & 0.0035 & 2.101 & 0.0024  \\
		
		$1 \times  10^5$ & 5.158 & 1.990 & 3.901 & 0.0024 & 16.781 & 0.019 \\
		\hline
	\end{tabular}
	\caption{Results from standard Monte Carlo and the importance
          sampling algorithms. Note that for ease of presentation the payoff and error are all scaled to be $10^{-9}$ of the values presented.}
	\label{Table:Results for Tanh}
\end{table}
The results in Table \ref{Table:Results for Tanh} highlight the key
differences in the algorithms. Clearly this is a difficult problem for
standard Monte Carlo to solve. The reason of course being that although $G$ is mollified it still changes value quickly over a small interval. For example $G$ at $0.25$ is approximately $10^{-10}$, but $G(0.5) \approx 10^{-7}$ and $G(0.75)\approx 10^{-4}$, hence a reasonably small change in the value of the SDE can influence the outcome significantly. However, for the standard Monte Carlo run, only $60$ of the $100,000$ were $>1/2$ at the terminal time and none were above $3/4$. Hence standard Monte Carlo is not giving much information about the most important region of the function.

The importance sampling schemes again give reduced errors,
however, this example highlights the differences between them. Although the complete measure change does have a smaller error than standard Monte Carlo the payoff oscillates around and hence the decoupled algorithm appears to be superior since the payoffs are consistent and the error decreases in the expected manner. 

$\triangleright$ \emph{Robustness of complete measure change.}
The above table shows why one has to consider the effect of the
measure change on the propagation of chaos error. The reason this is
more prominent here than in the previous example is because the
magnitude of the optimal measure change is far larger. Hence, even
when we use a large number of particles they  may provide a poor approximation of the law, this is where this algorithm lacks robustness.

\begin{remark}
	[Requirement for improved simulation]
	It is clear from these examples that combining importance sampling with MV-SDEs can provide a major reduction in the required computational cost, namely we can achieve a smaller variance with far less simulations (and hence time). When using decoupling, unfortunately one has to approximate the law first, which is computationally expensive to do using a particle approximation. Hence, one may look towards more sophisticated simulation techniques to speed up the first run, for example \cite{GobetPagliarani2018} or towards multilevel Monte Carlo such as \cite{SzpruchTanTse2017}. However, with the ability to almost eliminate the variance one should always keep in mind the benefits from importance sampling.
\end{remark}

%
%
%
%

\section{Proof of Main Results}
\label{Sec:Proofs}
We now provide the proofs of our two main
theorems. Throughout we work under the $\bP$-measure and we omit it as
a superscript in our Brownian motions. Some arguments align with those
of \cite{GuasoniRobertson2008} and we quote them where appropriate.

\subsection{Proofs for Theorem \ref{Thm:Complete Measure Change}}
\label{Sec:Proofs Complete Measure Change}

Continuity of the SDE w.r.t. Brownian motion is key as it allows to apply directly the contraction principle transferring Schilder's LDP for the Brownian motion to an LDP for the solution of the SDE; otherwise difficulties would arise when using Varadhan's lemma. Unlike the decoupled case, we will stick to Lipschitz coefficients here, the reason for this is that Lemma \ref{Lem:Bounding Xbar} does not generalise well for SDEs of the type \eqref{Eq:Form of MV-SDE}. 
\begin{lemma}
	\label{Lem:Continuity of Interacting SDE}
Fix $N \in \bN$, let Assumption \ref{Ass:Drift Lipschitz Assumption} hold and let $X \in \bS^{p}$ for $p \ge 2$ denote the $N$-dimensional strong solution to the SDE system defined in \eqref{Eq:Form of MV-SDE}. Then $X$ is continuous w.r.t. the set of $N$ Brownian motions in the uniform topology.
\end{lemma}

\begin{proof}
	 To show continuity in the uniform topology we consider two sets of iid Brownian motions, $W_{t}=(W_{t}^{1}, \dots, W_{t}^{N})$ and $\tilde{W}_{t}= (\tilde{W}_{t}^{1}, \dots, \tilde{W}_{t}^{N})$ and show continuity by analyzing the difference between, $\tilde{X}_{t}^{i} := X_{t}^{i}(\tilde{W}_{t}^{1}, \dots, \tilde{W}_{t}^{N})$ and $X_{t}^{i}$ with $i \in\{1,\cdots,N\}$. We have,
	\begin{align*}
	| \tilde{X}_{t}^{i,N} - X_{t}^{i,N}|
	\le
	\int_{0}^{t} | b(s, \tilde{X}_{s}^{i,N}, \frac{1}{N} \sum_{j=1}^{N} \delta_{\tilde{X}_{s}^{j,N}}) - b(s, X_{s}^{i,N}, \frac{1}{N} \sum_{j=1}^{N} \delta_{X_{s}^{j,N}}) | \dd s
	+
	| \int_{0}^{t} \sigma \dd \tilde{W}_{s}^{i} - \int_{0}^{t} \sigma \dd W_{s}^{i} | \, .
	\end{align*} 
	Considering the time integral first, we can bound as follows,
	\begin{align*}
	&\Big| b(s, \tilde{X}_{s}^{i,N}, \frac{1}{N} \sum_{j=1}^{N} \delta_{\tilde{X}_{s}^{j,N}}) - b(s, X_{s}^{i,N}, \frac{1}{N} \sum_{j=1}^{N} \delta_{X_{s}^{j,N}}) \Big|
\le
	C \Big( | \tilde{X}_{s}^{i,N} - X_{s}^{i,N}| + \Big( \frac{1}{N} \sum_{j=1}^{N} ( \tilde{X}_{s}^{j,N} - X_{s}^{j,N})^{2} \Big)^{\frac12} \Big) \, ,
	\end{align*}
	where we used the Lipschitz property and the definition of the Wasserstein-$2$ metric for empirical distributions (see \cite{BerntonEtAl2017}, for example). Noting that for the second term,
	\begin{align*}
	\Big( \frac{1}{N} \sum_{j=1}^{N} ( \tilde{X}_{s}^{j,N} - X_{s}^{j,N})^{2} \Big)^{\frac12}
	\le
	\max_{j \in \{1, \dots, N\}} | \tilde{X}_{s}^{j,N} - X_{s}^{j,N}|
	\le
	\sum_{j=1}^{N} | \tilde{X}_{s}^{j,N} - X_{s}^{j,N}| \, .
	\end{align*}
	Hence we can bound the drift by terms of the form $| \tilde{X}_{s}^{j,N} - X_{s}^{j,N}|$.  
This yields the following,
	\begin{align*}
	| \tilde{X}_{t}^{i,N} - X_{t}^{i,N}|
	\le &
	\int_{0}^{t} C \Big( | \tilde{X}_{s}^{i,N} - X_{s}^{i,N}| + \sum_{j=1}^{N} | \tilde{X}_{s}^{j,N} - X_{s}^{j,N}| \Big) \dd s
	 +
	C \sup_{0 \le s \le t} | \tilde{W}_{s}^{i} - W_{s}^{i}| \, .
	\end{align*} 
	Taking supremums and summing over $i$ on both sides yields,
	\begin{align*}
	\sum_{i=1}^{N} \sup_{0 \le t \le T} | \tilde{X}_{t}^{i,N} - X_{t}^{i,N}|
	\le &
	\int_{0}^{T} C \sum_{i=1}^{N} \sup_{0 \le t \le s} | \tilde{X}_{t}^{i,N} - X_{t}^{i,N}| \dd s +
	C \sum_{i=1}^{N} \sup_{0 \le s \le T}  | \tilde{W}_{s}^{i} - W_{s}^{i}|
	\\
	\le &
	C e^{CT} \sum_{i=1}^{N} \sup_{0 \le s \le T}  | \tilde{W}_{s}^{i} - W_{s}^{i}| \, ,
	\end{align*} 
	where the final step follows from Gronwall's inequality.
					It is then clear that $\sum_{i=1}^{N} \sup_{0 \le s \le T}  | \tilde{W}_{s}^{i} - W_{s}^{i}| \rightarrow 0$ implies $\sum_{i=1}^{N} \sup_{0 \le t \le T} | \tilde{X}_{t}^{i,N} - X_{t}^{i,N}| \rightarrow 0$, hence we obtain the required continuity.
\end{proof}

We next show that one can use Varadhan's lemma in this case.
\begin{lemma}
	\label{Lem:Particle Varadhan UI}
	Fix $N\in \bN$, let $h \in \bH_T$ and let Assumptions \ref{Ass:General Payoff Growth} and \ref{Ass:Drift Lipschitz Assumption} hold. 
	
	Then the integrability condition in Varadhan's lemma holds for \eqref{Eq:Small Noise for Particles}. Namely for $\gamma>1$
	\begin{align*}
	&
	\limsup_{\epsilon \rightarrow 0} 
	\epsilon \log \left(
	\bE_{\bP} \left[
	\exp \left(
	\frac{\gamma}{\epsilon} \left(
	2 \log \Big(\tilde{G}_{1}(\sqrt{\epsilon}W^{1}, \dots, \sqrt{\epsilon}W^{N})\Big)
	\int_{0}^{T} \sqrt{\epsilon} \dot{h}_{t}\dd W_{t}^{1}
	+
	\frac{1}{2} \int_{0}^{T} (\dot{h}_{t})^{2} \dd t
	\right)
	\right)
	\right]
	\right)
	< \infty.
	\end{align*}
\end{lemma}
\begin{proof}	
Using that $h\in \bH_T$  is deterministic, $\dot h \in L^{2}([0,T],\bR^N)$ and Cauchy-Schwarz  we obtain,
	\begin{align*}
	&\epsilon \log \left(
	\bE_{\bP} \left[
	\exp \left(
	\frac{\gamma}{\epsilon} \Big(
	2 \log (\tilde{G}_{1}(\sqrt{\epsilon}W^{1}, \dots, \sqrt{\epsilon}W^{N}))
	-
	\int_{0}^{T} \sqrt{\epsilon} \dot{h}_{t}\dd W_{t}^{1} 
	+
	\frac{1}{2} \int_{0}^{T} (\dot{h}_{t})^{2} \dd t
	\Big)
	\right)
	\right]
	\right)
	\\
	&
	\qquad \le
	\frac{\gamma}{2} \int_{0}^{T} (\dot{h}_{t})^{2} \dd t
	+ 
	\frac{\epsilon}{2} \log \left(
	\bE_{\bP} \left[
	\exp \left(
	\frac{4\gamma}{\epsilon}
	\log (\tilde{G}_{1}(\sqrt{\epsilon}W^{1}, \dots, \sqrt{\epsilon}W^{N}))
	\right)
	\right]
	\right)
	\\
	& \hspace{7cm} +
	\frac{\epsilon}{2} \log \left(
	\bE_{\bP} \left[
	\exp \left(-
	\frac{2\gamma}{\epsilon}
	 \int_{0}^{T} \sqrt{\epsilon} \dot{h}_{t} \dd W_{t}^{1}
	\right)
	\right]
	\right)
	\, .
	\end{align*}
It is then sufficient to show that the three terms are finite when we take $\limsup_{\epsilon \rightarrow 0} $. The first term is clearly finite by the conditions on $h$. 
	Finiteness of the third term follows from \cite{GuasoniRobertson2008}*{pg.16}, namely $\forall ~ i \in \{1, \dots, N\}$ the stochastic integral has the distribution $\int_{0}^{T} \dot h_{t} \dd W_{t}^{i} \sim \cN(0,\int_{0}^{T} \big(\dot h_{t})^{2} \dd t\big)$. Thus we obtain,
\begin{align*}
	\limsup_{\epsilon \rightarrow 0} 
	\frac{\epsilon}{2} \log \left(
	\bE_{\bP} \left[
	\exp \left(-
	\frac{2\gamma}{\epsilon}
	\int_{0}^{T} \sqrt{\epsilon} \dot{h}_{t} \dd W_{t}^{1}
	\right)
	\right]
	\right)
	=
\gamma^{2}
\int_{0}^{T} (\dot h_{t})^{2} \dd t
<
\infty \, .
\end{align*}
The final term to consider is the terminal term, $\log (\tilde{G}_{1})$. By definition of $\tilde{G}_{1}$ and Assumption \ref{Ass:General Payoff Growth} we have,
\begin{align*}
\log \left( \tilde{G}_{1}( \sqrt{\epsilon} W^{1}, \dots, \sqrt{\epsilon} W^{N}) \right)
\le
C_{1} + C_{2} \sup_{0 \le t \le T} |  X^{1,N}( \sqrt{\epsilon} W^{1}, \dots, \sqrt{\epsilon} W^{N})|^{\alpha} \, .
\end{align*}
Applying similar arguments as in Lemma \ref{Lem:Continuity of Interacting SDE} we obtain
\begin{align*}
|X_{t}^{1,N}&( \sqrt{\epsilon} W^{1}, \dots, \sqrt{\epsilon} W^{N})|
\\
&
\le 
C +\int_{0}^{t} C \Big( |X_{s}^{1,N}( \sqrt{\epsilon} W^{1}, \dots, \sqrt{\epsilon} W^{N})| + \sum_{j=1}^{N} | X_{s}^{j,N}( \sqrt{\epsilon} W^{1}, \dots, \sqrt{\epsilon} W^{N})| \Big) \dd s
+C \sqrt{\epsilon} \sup_{0 \le s \le t} | W_{s}^{1}| \, .
\end{align*} 
Noting that for $\alpha \ge 1$ and $a_{i}$ nonnegative,
$(\sum_{i=1}^{N} a_{i})^{\alpha} \le C^{\alpha} \sum_{i=1}^{N}
a_{i}^{\alpha}$ and that the  above estimate is true for any $X^{i,N}$, then taking supremums and summing over $i$ yields,
\begin{align*}
\sum_{i=1}^{N} \sup_{0 \le t \le T} & |X_{t}^{i,N}( \sqrt{\epsilon} W^{1}, \dots, \sqrt{\epsilon} W^{N})|^{\alpha}
\\
&
\le 
C^{\alpha} + \int_{0}^{T} C^{\alpha} \sum_{i=1}^{N} \sup_{0 \le t \le s} | X_{t}^{i,N}( \sqrt{\epsilon} W^{1}, \dots, \sqrt{\epsilon} W^{N})|^{\alpha} \dd s
+C^{\alpha} \sqrt{\epsilon}^{\alpha}
\sum_{i=1}^{N} \sup_{0 \le s \le T}  |W_{s}^{i}|^{\alpha}
\\
&
\leq
C^{\alpha} e^{C^{\alpha}} \sqrt{\epsilon}^{\alpha} 
\sum_{i=1}^{N} \sup_{0 \le s \le T}  |W_{s}^{i}|^{\alpha} \, ,
\end{align*} 
where the final line comes from Gronwall's inequality. It is useful for us to note this yields the bound
\begin{align}
\label{Eq:Tilde G bound}
\log \left( \tilde{G}_{1}( W^{1}, \dots, W^{N}) \right)
\le
C_{1} + C_{2}
\sum_{i=1}^{N} \sup_{0 \le s \le T}  |W_{s}^{i}|^{\alpha} \, .
\end{align}
Using the previous results we have the following bound,
\begin{align*}
&\frac{\epsilon}{2} \log \left(
\bE_{\bP} \left[
\exp \left(
\frac{4\gamma}{\epsilon}
\log (\tilde{G}_{1}(\sqrt{\epsilon}W^{1}, \dots, \sqrt{\epsilon}W^{N}))
\right)
\right]
\right)
\\
& \qquad  \qquad
\le
C + \sum_{i=1}^{N} \frac{\epsilon C}{2} \log \left(
\bE_{\bP} \left[
\exp \left(
\frac{4\gamma C^{\alpha}}{\epsilon^{1- \alpha/2}}
\sup_{0 \le s \le T}  |W_{s}^{i}|^{\alpha}
\right)
\right]
\right) \, ,
\end{align*}
where we have used the independence of the Brownian motions to obtain the sum over $i$.

Finiteness of this term then follows by arguments similar to those in Lemma 7.6 and 7.7 in \cite{GuasoniRobertson2008}. To conclude, we have shown that all terms are finite and the result follows.
\end{proof}

Before finishing the proof of Theorem \ref{Thm:Complete Measure Change}, we note that the LDP for Brownian motion in pathspace is given by Schilder's theorem, which states that for a $d$-dimensional Brownian motion $W$, then $\sqrt{\epsilon}W$ satisfies a LDP with good rate function (see \cite{DemboZeitouni2010}),
\begin{align*}
I(y)=
\begin{cases}
\frac{1}{2} \int_{0}^{T} |\dot{y}_{t}|^{2} \dd t \, , & \text{ if} ~ y \in \bH_{T}^{d} \, ,
\\
\infty, & \text{ if} ~ y \in \bW_{T}^{d} \backslash \bH_{T}^{d} \, .
\end{cases}
\end{align*}

\begin{proof}[Proof of Theorem \ref{Thm:Complete Measure Change}]
The continuity of the SDE from Lemma \ref{Lem:Continuity of Interacting SDE} along with existence of a unique strong solution under Assumptions \ref{Ass:Drift Lipschitz Assumption}, ensure $\tilde{G}_{1}$ is a continuous function under Assumption \ref{Ass:General Payoff Growth}. By assumption, there exists a point $(u^{1}, \hat{u}) \in \bH_{T}^{2}$ such that $\tilde{G}(u^{1}, \hat{u}, \dots, \hat{u})>0$, this along with \eqref{Eq:Tilde G bound} and recalling $\alpha < 2$ we obtain the existence of maximisers by Lemma 7.1 of \cite{GuasoniRobertson2008}. Similarly the $+ \dot{h}^{2}$ yields existence of a minimising $h$ for $\bar{L}$.

Moreover, continuity of $\tilde{G}$ w.r.t. the Brownian motion and finite variation of $\dot{h}$ implies the exponential term in \eqref{Eq:Small Noise for Particles} is continuous. Thus to use Varadhan's lemma we only need to check the integrability condition, which is given in Lemma \ref{Lem:Continuity of Interacting SDE}, hence relation \eqref{Eq:Complete L in sup form} follows.

The remaining part to be proved is that \eqref{Eq:Complete Asymptotic Optimality check} implies asymptotically optimal. This essentially relies on showing that \eqref{Eq:Complete L minimiser} is a lower bound for the RHS of \eqref{Eq:General Asymptotically Optimal Condition}. Using the same arguments to derive \eqref{Eq:General Asymptotically Optimal Condition}, one obtains the following expression for an asymptotically optimal estimator
\begin{align*}
\sup_{u \in \bH^N_{T}} \left\{
2 \log (\tilde{G}_{1}(u^{1}, \dots, u^{N}))
-
\frac{1}{2} \int_{0}^{T} |\dot u_{t}|^{2} \dd t
\right\} 
\, .
\end{align*}
It is then clear that the supremum is bounded below by the case $u^{2}= \dots = u^{N}$, which yields the expression \eqref{Eq:Complete L minimiser}.

	Strict convexity along with arguments on page 18 in \cite{GuasoniRobertson2008} yields the uniqueness which completes the proof.
\end{proof}

\subsection{Proofs for Theorem \ref{Thm:Decoupled}}
\label{Sec:Proofs Decoupled}

We recall, that due to the independence of the original particle system from the SDE in question, we work on the product of two probability spaces, consequently (since $\mu^{N}$ will be a ``realisation'' coming from the space $\tilde{\Omega}$) our results are all $\tilde{\bP}$-a.s..

As before we need to prove that the SDE is a continuous map of the Brownian motions. We were unable to find any results for the one-sided Lipschitz and locally Lipschitz case, we therefore provide a proof of this result here (Lemma \ref{Lem:Continuity of SDE}). The proof of this relies on the following lemma.

\begin{lemma}
	\label{Lem:Bounding Xbar}
Let Assumption \ref{Ass:Drift Monotone Assumption} hold and let $\bar{X}$ be the solution to \eqref{Eq:Particle Approx General MV-SDE}. Then consider the following stochastic processes
\begin{align*}
X_{t}^{+} & :=
x_{0} \1_{\{x_{0} \ge 0 \}} + \int_{0}^{t} C (|X_{s}^{+}| +1) \dd s + \sigma \left( \sup_{0 \le s \le t} W_{s} - \inf_{0 \le s \le t} W_{s} \right),
\\
X_{t}^{-} & :=
x_{0} \1_{\{x_{0} \le 0 \}} - \int_{0}^{t} C (|X_{s}^{-}| +1) \dd s + \sigma \left( \inf_{0 \le s \le t} W_{s} -  \sup_{0 \le s \le t} W_{s} \right) \, ,
\end{align*}
where $C$ is the constant in the monotone condition of $b$. 

Then, $\forall ~ t \ge 0$, $X_{t}^{-} \le \bar{X}_{t} \le X_{t}^{+}$, $\bP\otimes \tilde{\bP}$-a.s..
\end{lemma}

\begin{proof}
Firstly, one can easily show through a standard Picard iteration argument that both $X^{\pm}$ have unique, progressively measurable solutions in $\bS^{2}$.
We argue by contradiction and show the upper bound $\bar{X} \le X^{+}$, the lower bound follows by the same argument in the opposite direction.
Since $b$ is monotone (Assumption \ref{Ass:Drift Monotone Assumption}), we can derive the following bounds $\forall ~ s \in [0,T]$ and $\mu \in \cP_{2}(\bR)$,
\begin{align*}
b(s,x,\mu) \le C(|x|+1) \quad \text{for  } x \ge 0 
\qquad
\text{and}
\qquad
b(s,x,\mu) \ge -C(|x|+1) \quad \text{for  } x \le 0 \, .
\end{align*} 
Assume that there exists a time $t_{2}$ such that $\bar{X}_{t_{2}} >
X_{t_{2}}^{+}$. If $\bar{X}_{t} \ge 0$ for all $t \in [0,t_2]$, then,
\begin{align*}
X_{t_2}^{+}-\bar{X}_{t_2} =
x_{0} \1_{\{x_{0} \ge 0 \}} -x_{0}+ \int_{0}^{t_2} C (|X_{s}^{+}| +1) -b(s,\bar{X}_{s}, \mu_{s}^{N}) \dd s + \sigma \left( \sup_{0 \le s \le t_2} W_{s} - \inf_{0 \le s \le t_2} W_{s} \right) - \sigma W_{t_2} \geq0,
\end{align*}
which yields a contradiction. Alternatively, let $t_1:=\max\{t\leq
t_2: \bar X_t = 0\}$. By continuity, $\bar{X}_{t_{1}}=0$ and so
\begin{align*}
X_{t_{2}}^{+}-\bar{X}_{t_{2}}
 =
x_{0} \1_{\{x_{0} \ge 0 \}}+ & \int_{0}^{t_{2}} C (|X_{s}^{+}| +1) \dd s  -  \int_{t_{1}}^{t_{2}} b(s,\bar{X}_{s}, \mu_{s}^{N}) \dd s
\\
&
 + \sigma \left( \sup_{0 \le s \le t_{2}} W_{s} - \inf_{0 \le s \le t_{2}} W_{s} \right) - \sigma \left( W_{t_{2}}- W_{t_{1}} \right)\geq0,
\end{align*}
which contradicts $\bar{X}_{t_{2}} > X_{t_{2}}^{+}$ and thus proves the result.

\end{proof}

One can now use this lemma to prove the following result.
\begin{lemma}
	\label{Lem:Continuity of SDE}
	Let $\bar{X}$ be defined as in \eqref{Eq:Particle Approx General MV-SDE}, with coefficients satisfying Assumption \ref{Ass:Drift Monotone Assumption}, then $\bar{X}$ is a $\bP\otimes \tilde{\bP}$-a.s. continuous map of Brownian motion in the uniform norm.
\end{lemma}

\begin{proof}
To prove this result we require that, if 
$\sup_{0 \le s \le t} |\tilde{W}_{s}-W_{s}| \rightarrow 0$, 
then $\sup_{0 \le s \le t} |\bar{X}_{s}(\tilde{W})-\bar{X}_{s}(W)| \rightarrow 0$. 
We note that we work on the uniform topology and hence we may assume
that all (a finite number of) Brownian motions are uniformly bounded on $[0,T]$.
Lemma \ref{Lem:Bounding Xbar}, implies that we can bound the value $\bar{X}$ takes by the processes $X_{\cdot}^{\pm}$. It is a straightforward application of Gronwall's Lemma to deduce,
	\begin{align*}
	X_{t}^{+} & \le
	\Big(x_{0} \1_{\{x_{0} \ge 0 \}} + Ct +\sigma \Big( \sup_{0 \le s \le t} W_{s} - \inf_{0 \le s \le t} W_{s} \Big) \Big) e^{Ct} \, ,
	\\
	X_{t}^{-} & \ge
	-\Big(|x_{0} \1_{\{x_{0} \le 0 \}}| + Ct +\sigma \Big | \inf_{0 \le s \le t} W_{s} -\sup_{0 \le s \le t} W_{s} \Big | \Big) e^{Ct} \, .
	\end{align*}
	Hence we can bound the value $\bar{X}$ can take as a function of its Brownian motion (which itself is bounded by the uniform topology). Let us now consider the difference in the SDEs driven by the different Brownian motions,
	\begin{align*}
	| \bar{X}_{t}(\tilde{W})- \bar{X}_{t}(W) |
	\le
	\int_{0}^{t} | b(s, \bar{X}_{s}(\tilde{W}), \mu_{s}^{N}) - b(s, \bar{X}_{s}(W), \mu_{s}^{N})| \dd s + \sigma | \tilde{W}_{t} - W_{t}| \, .
	\end{align*} 
	By Assumption \ref{Ass:Drift Monotone Assumption}, $b$ is locally Lipschitz, hence,
	\begin{align*}
	| b(s, \bar{X}_{s}(\tilde{W}), \mu_{s}^{N}) - b(s, \bar{X}_{s}(W), \mu_{s}^{N})|
	\le
	C(\tilde{W},W) |\bar{X}_{s}(\tilde{W})- \bar{X}_{s}(W) | \, .
	\end{align*}
 Noting further that $\sigma | \tilde{W}_{t} - W_{t}| \le \sigma \sup_{0 \le s \le t} | \tilde{W}_{s} - W_{s}| $, then by Gronwall's inequality we obtain,
	\begin{align*}
	| \bar{X}_{t}(\tilde{W})- \bar{X}_{t}(W) |
	\le
	\sigma \Big( \sup_{0 \le s \le t} | \tilde{W}_{s} - W_{s}| \Big)
	e^{C(\tilde{W}, W)t}\, .
	\end{align*} 
	Again, by the uniform topology, we must have $\tilde{W}$ and $W$ bounded, thus $C(\tilde{W},W) < \infty$ and hence, $\sup_{0 \le s \le t} |\bar{X}_{s}(\tilde{W})-\bar{X}_{s}(W)| \rightarrow 0$ when $\sup_{0 \le s \le t} |\tilde{W}_{s}-W_{s}| \rightarrow 0$.	
	
\end{proof}

We now prove that the uniform integrability condition still holds, namely that we can still apply Varadhan's Lemma, in both settings.

\begin{lemma}
	\label{Lem:UI condition decoupled}
	Let $h\in \bH_T$, then under Assumption \ref{Ass:General Payoff Growth} and \ref{Ass:Drift Monotone Assumption} the integrability condition in Varadhan's lemma holds for \eqref{Eq:Small Noise for General MV-SDE}. Namely, for some $\gamma >1$
	\begin{align*}
	\limsup_{\epsilon \rightarrow 0}
	\epsilon \log \bE_{\bP \otimes \tilde{\bP}} \left[ 
	\exp\left( \frac{\gamma}{\epsilon}\left(2 \log(\overline{G}(\sqrt{\epsilon} W)) - \int_{0}^{T} \sqrt{\epsilon} \dot{h}_{t} \dd W_{t} + \frac{1}{2} \int_{0}^{T} \dot{h}_{t}^{2} \dd t \right) \right)
	\Big |
	\tilde{\cF}
	\right] < \infty
	~ ~ ~ \tilde{\bP} \text{-a.s.}.
	\end{align*}
\end{lemma}

\begin{proof}
	The $h$ terms can be dealt with using the same arguments as before.
	The term we are interested in is the $G$ term. Using arguments as in the proof of Lemma \ref{Lem:Particle Varadhan UI}, we only need to prove the following holds,
	\begin{align*}
	\limsup_{\epsilon \rightarrow 0}
	\frac{\epsilon}{2} \log \left(
	\bE_{\bP\otimes \tilde{\bP}} \left[
	\exp \left(
	\frac{4\gamma}{\epsilon}
	\log \Big(G(\bar{X}(\sqrt{\epsilon}W)) \Big)
	\right)
	\Big |
	\tilde{\cF}
	\right]
	\right)
	< \infty \, .
	\end{align*}

 
Recall that Lemma \ref{Lem:Bounding Xbar}, yields the bound, $X_{t}^{-} \le \bar{X}_{t} \le X_{t}^{+}$, $\bP \otimes \tilde{\bP}$-a.s.. Hence, for $\alpha \in [1, 2)$ we have the following bound $\bP \otimes \tilde{\bP}$-a.s.,
\begin{align*}
\sup_{0 \le t \le T} | \bar{X}_{t}|^{\alpha}
\le
\sup_{0 \le t \le T} | X_{t}^{+}|^{\alpha}
+ \sup_{0 \le t \le T} | X_{t}^{-}|^{\alpha}
=
| X_{T}^{+}|^{\alpha}
+ | X_{T}^{-}|^{\alpha} \, ,
\end{align*}
where the final equality comes from the fact $|X^{\pm}|$ are
nondecreasing processes. Due to the dependence on the external measure
$\mu^{N}$, all of these results are $\tilde{\bP}$-a.s., but for ease
of presentation we will omit it here. Further recall that by Gronwall's lemma (or see proof of Lemma \ref{Lem:Continuity of SDE}), we can bound the processes $|X^{\pm}|$, thus,
\begin{align*}
|X_{T}^{+}|^{\alpha} 
& \le
C^{\alpha}\Big(x_{0}^{\alpha} \1_{\{x_{0} \ge 0 \}} + C^{\alpha} +\sigma^{\alpha} \Big( \sup_{0 \le s \le T} W_{s} - \inf_{0 \le s \le T} W_{s} \Big)^{\alpha} \Big) e^{C \alpha} \, ,
\\
|X_{T}^{-}|^{\alpha} 
& \le
C^{\alpha}\Big(|x_{0} \1_{\{x_{0} \le 0 \}}|^{\alpha} + C^{\alpha} +\sigma^{\alpha} \Big | \inf_{0 \le s \le T} W_{s} -\sup_{0 \le s \le T} W_{s} \Big |^{\alpha} \Big) e^{C \alpha } \, .
\end{align*}
Due to the fact that $\alpha \ge 1$, and $- \inf_{0 \le s \le T} W_{s}= \sup_{0 \le s \le T} -W_{s} \ge 0$, we have,
\begin{align*}
\Big | \inf_{0 \le s \le T} W_{s} -\sup_{0 \le s \le T} W_{s} \Big |^{\alpha}
=
\Big( \sup_{0 \le s \le T} W_{s} - \inf_{0 \le s \le T} W_{s} \Big)^{\alpha}
\le
C^{\alpha} \Big( \Big(\sup_{0 \le s \le T} W_{s}\Big)^{\alpha} 
+ \Big(\sup_{0 \le s \le T} -W_{s} \Big)^{\alpha} \Big) \, .
\end{align*}
We express the bound w.r.t. the driving Brownian motion $\sqrt{\epsilon}W$ and obtain,
\begin{align*}
\sup_{0 \le t \le T} | \bar{X}_{t}(\sqrt{\epsilon}W)|^{\alpha}
\le
C^{\alpha}\Big(
|x_{0}|^{\alpha} + C^{\alpha} +C^{\alpha}\sigma^{\alpha} \sqrt{\epsilon}^{\alpha} \Big( \Big(\sup_{0 \le s \le T} W_{s}\Big)^{\alpha} 
+ \Big(\sup_{0 \le s \le T} -W_{s} \Big)^{\alpha} \Big) \Big) e^{C \alpha} \, .
\end{align*}
We can simplify this further by noting,
\begin{align*}
\Big(\sup_{0 \le s \le T} W_{s}\Big)^{\alpha} 
+ \Big(\sup_{0 \le s \le T} -W_{s} \Big)^{\alpha}
\le
C^{\alpha}\sup_{0 \le s \le T} |W_{s}|^{\alpha} \, .
\end{align*}
	Using these inequalities we obtain,
		\begin{align*}
	&\frac{\epsilon}{2} \log \left(
	\bE_{\bP \otimes \tilde{\bP}} \left[
	\exp \left(
	\frac{4\gamma}{\epsilon}
	\log (G(\bar{X}(\sqrt{\epsilon}W)))
	\right)
	\Big |
	\tilde{\cF}
	\right]
	\right)
	\\
	& \qquad \le
	\frac{\epsilon}{2} \log \left(
	\bE_{\bP \otimes \tilde{\bP}} \left[
	\exp \left(
	\frac{4\gamma}{\epsilon}C_{1}
	+
	\frac{4\gamma}{\epsilon}C_{2}
	\Big(
	C^{\alpha}\Big(|x_{0}|^{\alpha} + C^{\alpha} +C^{\alpha}\sigma^{\alpha} \sqrt{\epsilon}^{\alpha} \sup_{0 \le s \le T} |W_{s}|^{\alpha} \Big) e^{C \alpha}
	\Big)
	\right)
	\Big |
	\tilde{\cF}
	\right]
	\right)
	\, .
	\end{align*}
By splitting up the terms in the exponential this then reduces to the problem of considering,
\begin{align*}
\frac{\epsilon}{2} \log \left(
\bE_{\bP \otimes \tilde{\bP}} \Big[
\exp \Big(
\frac{4\gamma}{\epsilon^{1- \alpha/2}}C_{2}
C^{\alpha}\sigma^{\alpha} \sup_{0 \le s \le T} |W_{s}|^{\alpha}
\Big)
\Big |
\tilde{\cF}
\Big]
\right)
\, .
\end{align*}	
One can show that this quantity is finite by following the same arguments as \cite{GuasoniRobertson2008}*{pg.16}.	
\end{proof}

We can now prove the second main theorem, the arguments follow similar lines to those we used to conclude the proof of Theorem \ref{Thm:Complete Measure Change}.
\begin{proof}[Proof of Theorem \ref{Thm:Decoupled}]
The continuity of the SDE from Lemma \ref{Lem:Continuity of SDE} along with existence of a unique strong solution under Assumption \ref{Ass:Drift Monotone Assumption}, ensure $\overline{G}$ is a $\tilde{\bP}$-a.s. continuous function under Assumption \ref{Ass:General Payoff Growth}. We then obtain the existence of the maximiser by Lemma 7.1 of \cite{GuasoniRobertson2008}.

Moreover, the $\tilde{\bP}$-a.s. continuity of $\overline{G}$ w.r.t. the Brownian motion and finite variation of $\dot{h}$ implies that to use Varadhan's lemma we only need to check the integrability condition, which is given in Lemma \ref{Lem:UI condition decoupled}. This with Lemma 7.6 in \cite{GuasoniRobertson2008} is enough to complete the proof by arguments on page 18 in \cite{GuasoniRobertson2008}.
\end{proof}


\begin{bibdiv}
\begin{biblist}

\bib{BudhirajaDupuisFischer2012}{article}{
      author={Budhiraja, Amarjit},
      author={Dupuis, Paul},
      author={Fischer, Markus},
       title={Large deviation properties of weakly interacting processes via
  weak convergence methods},
        date={2012},
        ISSN={0091-1798},
     journal={Ann. Probab.},
      volume={40},
      number={1},
       pages={74\ndash 102},
         url={https://doi.org/10.1214/10-AOP616},
}

\bib{BerntonEtAl2017}{article}{
      author={Bernton, Espen},
      author={Jacob, Pierre~E.},
      author={Gerber, Mathieu},
      author={Robert, Christian~P.},
       title={Inference in generative models using the {W}asserstein distance},
        date={2017},
     journal={arXiv:1701.05146},
}

\bib{BuckdahnEtAl2017}{article}{
      author={Buckdahn, Rainer},
      author={Li, Juan},
      author={Peng, Shige},
      author={Rainer, Catherine},
       title={Mean-field stochastic differential equations and associated
  {PDE}s},
        date={2017},
     journal={The Annals of Probability},
      volume={45},
      number={2},
       pages={824\ndash 878},
}

\bib{Bossy2004}{article}{
      author={Bossy, Mireille},
       title={Optimal rate of convergence of a stochastic particle method to
  solutions of 1d viscous scalar conservation laws},
        date={2004},
     journal={Mathematics of computation},
      volume={73},
      number={246},
       pages={777\ndash 812},
}

\bib{BossyTalay1997}{article}{
      author={Bossy, Mireille},
      author={Talay, Denis},
       title={A stochastic particle method for the {M}ckean-{V}lasov and the
  {B}urgers equation},
        date={1997},
     journal={Mathematics of Computation of the American Mathematical Society},
      volume={66},
      number={217},
       pages={157\ndash 192},
}

\bib{Carmona2016Lectures}{book}{
      author={Carmona, Ren\'{e}},
       title={Lectures of {BSDE}s, stochastic control, and stochastic
  differential games with financial applications},
   publisher={SIAM},
        date={2016},
}

\bib{CrisanMcMurray2017}{unpublished}{
      author={Crisan, Dan},
      author={McMurray, Eamon},
       title={Cubature on {W}iener space for {M}c{K}ean--{V}lasov {SDE}s with
  smooth scalar interaction},
        date={2017},
        note={arXiv:1703.04177},
}

\bib{DupuisEllis2011}{book}{
      author={Dupuis, Paul},
      author={Ellis, Richard~S.},
       title={A weak convergence approach to the theory of large deviations},
   publisher={John Wiley \& Sons},
        date={2011},
      volume={902},
}

\bib{DawsonGaertner1987-DG1987}{article}{
      author={Dawson, Donald~A.},
      author={G\"artner, J\"urgen},
       title={Large deviations from the {M}c{K}ean-{V}lasov limit for weakly
  interacting diffusions},
        date={1987},
        ISSN={0090-9491},
     journal={Stochastics},
      volume={20},
      number={4},
       pages={247\ndash 308},
         url={http://dx.doi.org/10.1080/17442508708833446},
}

\bib{dosReisSalkeldTugaut2017}{unpublished}{
      author={dos Reis, G.},
      author={Salkeld, William},
      author={Tugaut, Julian},
       title={Freidlin-{W}entzell {LDP}s in path space for {McK}ean-{V}lasov
  equations and the functional iterated logarithm law},
        date={2017},
        note={arXiv:1708.04961},
}

\bib{DupuisWang2004}{article}{
      author={Dupuis, Paul},
      author={Wang, Hui},
       title={Importance sampling, large deviations, and differential games},
        date={2004},
     journal={Stochastics: An International Journal of Probability and
  Stochastic Processes},
      volume={76},
      number={6},
       pages={481\ndash 508},
}

\bib{DemboZeitouni2010}{book}{
      author={Dembo, Amir},
      author={Zeitouni, Ofer},
       title={Large deviations techniques and applications, volume 38 of
  stochastic modelling and applied probability},
   publisher={Springer-Verlag, Berlin},
        date={2010},
}

\bib{EkelandTemam1999}{book}{
      author={Ekeland, Ivar},
      author={Temam, Roger},
       title={Convex analysis and variational problems},
   publisher={SIAM},
        date={1999},
}

\bib{Fischer2014}{article}{
      author={Fischer, Markus},
       title={On the form of the large deviation rate function for the
  empirical measures of weakly interacting systems},
        date={2014},
     journal={Bernoulli},
      volume={20},
      number={4},
       pages={1765\ndash 1801},
}

\bib{FlemingRishel1975}{book}{
      author={Fleming, Wendell~H.},
      author={Rishel, Raymond~W.},
       title={Deterministic and stochastic optimal control},
   publisher={Springer-Verlag, Berlin-New York},
        date={1975},
        note={Applications of Mathematics, No. 1},
}

\bib{GlassermanEtAl1999}{article}{
      author={Glasserman, Paul},
      author={Heidelberger, Philip},
      author={Shahabuddin, Perwez},
       title={Asymptotically optimal importance sampling and stratification for
  pricing path-dependent options},
        date={1999},
     journal={Mathematical finance},
      volume={9},
      number={2},
       pages={117\ndash 152},
}

\bib{GobetPagliarani2018}{article}{
      author={Gobet, Emmanuel},
      author={Pagliarani, Stefano},
       title={Analytical approximations of non-linear {SDE}s of
  {M}c{K}ean-{V}lasov type},
        date={2018},
     journal={Journal of Mathematical Analysis and Applications},
}

\bib{GuasoniRobertson2008}{article}{
      author={Guasoni, Paolo},
      author={Robertson, Scott},
       title={Optimal importance sampling with explicit formulas in continuous
  time},
        date={2008},
     journal={Finance and Stochastics},
      volume={12},
      number={1},
       pages={1\ndash 19},
}

\bib{GlassermanWang1997}{article}{
      author={Glasserman, Paul},
      author={Wang, Yashan},
       title={Counterexamples in importance sampling for large deviations
  probabilities},
        date={1997},
     journal={The Annals of Applied Probability},
      volume={7},
      number={3},
       pages={731\ndash 746},
}

\bib{KohatsuOgawa1997}{article}{
      author={Kohatsu-Higa, Arturo},
      author={Ogawa, Shigeyoshi},
       title={Weak rate of convergence for an {E}uler scheme of nonlinear
  {SDE}'s},
        date={1997},
     journal={Monte Carlo Methods and Applications},
      volume={3},
       pages={327\ndash 345},
}

\bib{KosturEtAl2002}{article}{
      author={Kostur, Marcin},
      author={{\L}uczka, J.},
      author={Schimansky-Geier, L.},
       title={Nonequilibrium coupled {B}rownian phase oscillators},
        date={2002},
     journal={Physical Review E},
      volume={65},
      number={5},
       pages={051115},
}

\bib{MitrinovicEtAl2012}{book}{
      author={Mitrinovic, Dragoslav~S.},
      author={Pecaric, Josip},
      author={Fink, Arlington~M.},
       title={Inequalities involving functions and their integrals and
  derivatives},
   publisher={Springer Science \& Business Media},
        date={2012},
      volume={53},
}

\bib{Robertson2010}{article}{
      author={Robertson, Scott},
       title={Sample path large deviations and optimal importance sampling for
  stochastic volatility models},
        date={2010},
     journal={Stochastic Processes and their applications},
      volume={120},
      number={1},
       pages={66\ndash 83},
}

\bib{SzpruchTanTse2017}{unpublished}{
      author={Szpruch, Lukasz},
      author={Tan, Shuren},
      author={Tse, Alvin},
       title={Iterative particle approximation for {M}ckean-{V}lasov {SDE}s
  with application to multilevel {M}onte {C}arlo estimation},
        date={2017},
        note={ArXiv:1706.00907},
}

\bib{TengEtAl2016}{article}{
      author={Teng, Huei-Wen},
      author={Fuh, Cheng-Der},
      author={Chen, Chun-Chieh},
       title={On an automatic and optimal importance sampling approach with
  applications in finance},
        date={2016},
     journal={Quantitative Finance},
      volume={16},
      number={8},
       pages={1259\ndash 1271},
}

\bib{YongZhou1999}{book}{
      author={Yong, Jiongmin},
      author={Zhou, Xun~Yu},
       title={Stochastic controls: Hamiltonian systems and {HJB} equations},
   publisher={Springer Science \& Business Media},
        date={1999},
      volume={43},
}

\end{biblist}
\end{bibdiv}


\end{document}